\documentclass{amsart}

\usepackage{amsmath}
\usepackage{amssymb}
\usepackage{comment}
\usepackage{hyperref}
\usepackage{xcolor}

\newtheorem{theorem}{Theorem}[section]
\newtheorem{lemma}[theorem]{Lemma}
\newtheorem{corollary}[theorem]{Corollary}
\newtheorem{proposition}[theorem]{Proposition}
\newtheorem{conj}[theorem]{Conjecture}

\theoremstyle{definition}
\newtheorem{definition}[theorem]{Definition}

\theoremstyle{remark}

\newcommand\R{\mathbb{R}}
\newcommand\Z{\mathbb{Z}}
\newcommand\T{\mathbb{T}}

\newcommand\N{\mathbb{N}}

\newcommand\op{\mathrm{op}}
\newcommand\hs{\mathfrak I_2}
\newcommand\tc{\mathfrak I_1}
\newcommand{\qtq}[1]{\quad\text{#1}\quad}
\newcommand\Schw{\mathcal{S}}
\newcommand\eps{\varepsilon}

\let\Re=\undefined\DeclareMathOperator{\Re}{Re}
\let\Im=\undefined\DeclareMathOperator{\Im}{Im}

\DeclareMathOperator{\tr}{tr}

\numberwithin{equation}{section}

\allowdisplaybreaks

\begin{document}

\title[On the well-posedness problem for DNLS]{On the well-posedness problem for the\\derivative nonlinear Schr\"odinger equation}

\author{Rowan Killip}
\address{Department of Mathematics, University of California, Los Angeles, CA 90095, USA}
\email{killip@math.ucla.edu}

\author{Maria Ntekoume}
\address{Department of Mathematics, Rice University, Houston, TX 77005-1892, USA}
\email{maria.ntekoume@rice.edu}

\author{Monica Vi\c{s}an}
\address{Department of Mathematics, University of California, Los Angeles, CA 90095, USA}
\email{visan@math.ucla.edu}

\begin{abstract}
We consider the derivative nonlinear Schr\"odinger equation in one space dimension, posed both on the line and on the circle.  This model is known to be completely integrable and $L^2$-critical with respect to scaling.

The first question we discuss is whether ensembles of orbits with $L^2$-equicontinuous initial data remain equicontinuous under evolution.  We prove that this is true under the restriction $M(q)=\int |q|^2 < 4\pi$. We conjecture that this restriction is unnecessary.

Further, we prove that the problem is globally well-posed for initial data in $H^{1/6}$ under the same restriction on $M$.  Moreover, we show that this restriction would be removed by a successful resolution of our equicontinuity conjecture. 
\end{abstract}

\maketitle

\section{Introduction} \label{sec;introduction}

The derivative nonlinear Schr\"odinger equation
\begin{equation} \label{DNLS} \tag{DNLS}
    i q_t + q'' +i \left(|q|^2 q\right)'=0
\end{equation}
describes the evolution of a complex-valued field $q$ defined either on the line $\R$ or the circle $\T=\R/\Z$.  This equation was introduced as an effective model in magneto-hydrodynamics; see \cite{IchikawaWatanabe, MOMT, mjolhus_1976}.  It was soon shown to be completely integrable \cite{MR464963} and has received enduring attention since that time.

As we shall document more fully below, well-posedness questions for \eqref{DNLS}, particularly \emph{global} well-posedness, have been particularly stubborn.  Local well-posedness is already very challenging: the nonlinearity contains a full derivative, like KdV or mKdV, while the linear part gives only Schr\"odinger-like smoothing.

The task of converting local into global well-posedness is typically a matter of exploiting conservation laws.  As a completely integrable system, \eqref{DNLS} has an infinite family of conserved quantities.  The first three are as follows:
\begin{align}
M(q) &= \int |q(x)|^2\,dx \\
H(q)&=-\tfrac{1}{2}\int i (q \bar q'-\bar q q')+|q|^4 \, dx \\
H_2(q)&=\int |q'|^2 +\tfrac{3}{4} i |q|^2 (q \bar q'-\bar q q') + \tfrac 12 |q|^6\, dx.
\end{align}

The striking fact about \eqref{DNLS} is that, with the exception of $M(q)$, none of the Hamiltonians in the hierarchy are coercive.  Indeed, algebraic solitons have $M=4\pi$ but all other Hamiltonians are identically zero.  Applying the scaling symmetry
\begin{align}\label{E:scaling}
q(t,x) \mapsto \sqrt{\lambda}\, q(\lambda^2 t,\lambda x)
\end{align}
to an algebraic soliton yields a one-parameter family of solutions with identical values for \emph{all} the conserved quantities. However, this family is unbounded in $H^s$ for every $s>0$. 

The quantity $H(q)$ serves as the Hamiltonian for \eqref{DNLS} with respect to the Poisson structure
\begin{align}\label{Poisson}
    \{F,G\}=\int \tfrac{\delta F}{\delta q} (\tfrac{\delta G}{\delta \bar q})'+ \tfrac{\delta F}{\delta \bar q} (\tfrac{\delta G}{\delta q})' \,dx,
\end{align}
while $M(q)$ generates translations, albeit at speed 2.  Although the momentum is given by $\tfrac12 M(q)$, our definition of $M$ leads to a more seamless connection to the existing literature.   

Given that $M(q)$ is invariant under both \eqref{DNLS} and the scaling \eqref{E:scaling}, it is natural to ask whether or not \eqref{DNLS} is well-posed in $L^2$.  This is not known.  Indeed, the existing local well-posedness theory requires $H^s$ initial data with $s\geq \frac12$.  (We will make some further progress on this question in this paper.)  It is important to recognize that because $M(q)$ is scaling critical, the mere fact that it forms a coercive conservation law would not suffice to render local well-posedness in $L^2$ automatically global.  One must fear the solution concentrates at one (or more) points in space, a scenario known as type-II blowup.  We do not believe this happens:  

\begin{conj}\label{petit}
For any $Q\subseteq \Schw$ that is $L^2$-bounded and equicontinuous, the totality of states reached by \eqref{DNLS} orbits originating from $Q$, that is,
\begin{align}\label{E:Conj1}
Q_* = \{ e^{tJ\nabla H} q : q\in Q \text{ and } t\in\R\}
\end{align}
is also $L^2$-equicontinuous.
\end{conj}

Here $\Schw$ denotes Schwartz class in the line case and $C^\infty$ on the torus.  In the line case, recent works (discussed below) guarantee that all such initial data lead to global Schwartz solutions.  The analogous claim is unknown on the torus, though we believe it to be true. Nevertheless, one can still ask if equicontinuity holds for as long as the orbits do exist.  By the arguments presented in this paper, solutions cannot break down without losing equicontinuity.  Therefore, a positive resolution of the conjecture for such partial solutions would already guarantee that they are global and so settle the conjecture in its entirety; see Corollary~\ref{C:H1}.

We phrased the conjecture in terms of $\Schw$ initial data because it is a class that is dense in all relevant spaces.  It also serves to emphasize that the central question to be addressed is not inherently tied to low regularity. 

Equicontinuity in $L^2$ is most easily understood via Fourier transformation: it means that $|\hat q|^2$ forms a tight family of measures.  Notice that, in view of the uncertainty principle, concentration on the physical side must be accompanied by a loss of tightness on the Fourier side.

In setting this conjecture, we have in mind four principal reasons: (1) It is challenging, yet recent developments give us hope for a successful resolution.  (2) It encapsulates a single essential obstacle, namely, understanding conservation laws for \eqref{DNLS}.  (3) A proof of this conjecture would have significant consequences for the well-posedness problem.  Indeed, such equicontinuity results form an essential part of a recent program developed in \cite{bringmann2019global,HGKV,KV18} that has proved successful in obtaining optimal well-posedness results for completely integrable PDE. (4) We are able to verify that it is true in the regime $M(q)<4\pi$; see Theorem~\ref{T:equi} below.

Given the nature of completely integrable systems, it is natural to imagine that an equicontinuity conjecture of the same form holds for all other PDE in the \eqref{DNLS} hierarchy.  Indeed, we truly believe that this is so and will shortly formulate just such a conjecture.  However, the particular claim that we believe will be of greatest use in understanding the hierarchy is best expressed through the perturbation determinant.  Let us turn our attention now to presenting this object, beginning with the requisite background.  

The Lax pair introduced by Kaup and Newell \cite{MR464963} for \eqref{DNLS} employs
\begin{align*}
   L_{\text{KN}}= 
   \begin{bmatrix}
   -i \lambda^2 -\partial & \lambda q\\
   -\lambda \bar q & i \lambda^2 -\partial
    \end{bmatrix}.
\end{align*}
For what follows, it will be convenient to make some cosmetic changes to this choice.  Specifically, we set $\lambda=e^{i\pi/4}\sqrt{\kappa}$ with $\kappa\geq 1$ and replace $e^{i\pi/4}q\mapsto q$.  This yields 
\begin{align*}
   L(\kappa) := \begin{bmatrix}  1 &0\\ 0 &-1\end{bmatrix}
   \begin{bmatrix}\kappa -\partial & \sqrt\kappa q\\
   i\sqrt\kappa \bar q & \kappa +\partial\end{bmatrix}
\ \ \ \text{and, for $q\equiv 0$,}\ \ \ %
    L_0(\kappa):=\begin{bmatrix} \kappa-\partial & 0\\ 0 & -(\kappa+\partial)
    \end{bmatrix}.
\end{align*}
These modifications maintain the crucial property that for smooth functions
\begin{align*}
\text{$q(t)$ solves \eqref{DNLS}} \iff \tfrac{d}{dt} L(t;\kappa) = \bigl[P(t;\kappa),\, L(t;\kappa)\bigr], 
\end{align*}
where 
\begin{align*}
P(\kappa)=\begin{bmatrix}
   2i \kappa^2-\kappa |q|^2 & 2i\kappa^{\frac32} q -\kappa^{\frac12} |q|^2 q +i\kappa^{\frac12} q'\\
   2\kappa^{\frac32} \bar q +i\kappa^{\frac12} |q|^2 \bar q -\kappa^{\frac12} \bar q' & -2i \kappa^2+\kappa |q|^2
    \end{bmatrix}.
\end{align*}
This guarantees that the Lax operators $L$ at different times are conjugate, at least formally.  This in turn suggests that the perturbation determinant $\det [ L_0^{-1}(\kappa) L(\kappa)] $ should be well-defined and conserved by the flow.

To make this precise, it is convenient for us to mimic the analysis of the AKNS-ZS system employed in \cite{KVZ17}:   Let us first define $(\kappa\pm\partial)^{-\frac 12}$ as the Fourier multipliers $(\kappa\pm i\xi)^{-\frac 12}$, where the complex square root is determined by $\sqrt{\kappa}>0$ and continuity.  We then define
\begin{align}\label{Lambda}
    \Lambda(q):=(\kappa-\partial)^{-\frac 12} q (\kappa+\partial)^{-\frac 12} \qtq{and}     \Gamma(q):=(\kappa+\partial)^{-\frac 12} \bar q (\kappa-\partial)^{-\frac 12},
\end{align}
which are Hilbert-Schmidt operators for $q\in L^2$; see Lemma~\ref{L:HS}.  Thus
\begin{align}\label{a defn}
a(\kappa;q) = \det\bigl[ 1 - i\kappa\Lambda\Gamma\bigr]
\end{align}
is well defined for $q\in L^2$ (and extends holomorphically to all $\Re \kappa>0$); moreover, for $q\in \Schw$ it agrees with the formal notion of the perturbation determinant.

While $a(\kappa)$ does encode all the Hamiltonians of the \eqref{DNLS} hierarchy, this is more easily seen through its logarithm,
\begin{align}\label{alpha defn}
\alpha(\kappa;q) := -\log[a(\kappa;q)]  = \sum_{\ell\geq 1} \tfrac{1}{\ell} \tr\left\{\left(i \kappa \Lambda \Gamma \right)^{\ell}\right\},
\end{align}
which serves as a generating function for these conservation laws.
Due to the possibility of $a(\kappa)$ vanishing, $\alpha(\kappa)$ may not be defined for all $\kappa\geq 1$.  Nevertheless, the series in \eqref{alpha defn} does converge for fixed $q\in L^2$ and $\kappa$ sufficiently large; see Proposition~\ref{P:large kappa}.

We have not yet addressed the conservation of $a(\kappa;q)$ under the \eqref{DNLS} flow.  In the line case, this could be effected by demonstrating that $a(\kappa;q)$ coincides with the reciprocal of the transmission coefficient and then appealing to the inverse scattering theory.  However, two direct proofs have appeared recently in the literature: Klaus and Schippa \cite{klaus2020priori} argued by differentiating the series (following a model introduced in \cite{KVZ17}), while Tang and Xu \cite{tang2020microscopic} developed a microscopic representation of this conservation law (in the style of \cite{HGKV}).  While these papers impose a small $M(q)$ requirement, this is solely to guarantee the convergence of the series \eqref{alpha defn}.  This issue is remedied by our Proposition~\ref{P:large kappa}.

To state the grand version of Conjecture~\ref{petit}, covering a wide range of commuting flows, let us first introduce a replacement for the set $Q_*$ defined in \eqref{E:Conj1}.
Given $q\in \Schw$, we first define
\begin{align}\label{Cq}
C_q= \{ \tilde q\in \Schw: a(\kappa;\tilde q) = a(\kappa; q) \text{ for all }\kappa>0\}
\end{align}
and write $C_q^0$ for the connected component (in the $L^2$ topology) of $C_q$ containing $q$.   Finally, given a set $Q\subseteq \Schw$, we define
\begin{align}\label{D:Q*}
Q_{**}=\bigcup_{q\in Q} C_q^0.
\end{align}

\begin{conj}\label{grand}
If $Q\subseteq \Schw$ is $L^2$-bounded and equicontinuous, then so too is the set $Q_{**}$ defined in \eqref{D:Q*}.
\end{conj}

We have several motivations in choosing connected components when defining $Q_{**}$.  This formulation of the conjecture retains a vestige of the behavior of orbits, while emphasizing that this is a question about conservation laws and is ultimately independent of the well-posedness of any flow.  Note also that while the zero solution and the family of algebraic solitons all share $a(\kappa)\equiv 1$, they are not in the same connected component under the \eqref{DNLS} hierarchy.

Our most compelling evidence in favor of these two conjectures is that both hold in the regime where $M(q)<4\pi$:

\begin{theorem}\label{T:equi}
Let $Q\subseteq \Schw$ be an $L^2$-equicontinuous set satisfying
\begin{align}\label{mass bound}
\sup\bigl\{ \| q\|_{L^2}^2 : q \in Q \bigr\} < 4\pi.
\end{align}
Then the set $Q_{**}$ defined in \eqref{D:Q*} is $L^2$-bounded and equicontinuous.
\end{theorem}

The significance of $4\pi$ is this: it is the value of $M$ at which the polynomial conservation laws lose their efficacy.  It is also the value of $M$ for the algebraic soliton, which is maximal among all solitary wave solutions.  Unlike mass-critical NLS, \eqref{DNLS} admits solitons of arbitrarily small $L^2$ norm and consequently, there is no notion of a scattering threshold.

The proof of Theorem~\ref{T:equi}, which will be given in Section~\ref{sec;equicontinuity}, is both short and simple.  Indeed, the hypothesis \eqref{mass bound} even allows us to forgo the restriction to connected components.

It has been observed before that $\tr(i\kappa\Lambda\Gamma)$ may be used to understand how the $L^2$ norm of $q$ is distributed across frequencies (cf. Lemma~\ref{L:dominant}).  The key observation that allows us to reach all the way to $4\pi$ (as opposed to mere smallness cf. \cite{klaus2020priori, tang2020microscopic}) is the manner in which we handle the remainder, specifically, the observation that the remainder may be summed in $\kappa$ for any $q\in L^2$; see \eqref{det vs tr}.

While the $4\pi$ restriction is crucial to our proof of Theorem~\ref{T:equi}, it does not play any role in our subsequent analysis of the consequences of such equicontinuity.  For this reason, we introduce a general threshold $M_*$:

\begin{definition}\label{D:M_*}
Let $M_*$ denote the maximal constant so that for any $L^2$-equi\-continuous set $Q\subseteq \Schw$ satisfying
\begin{align}\label{M_*}
\sup\bigl\{ \| q\|_{L^2}^2 : q \in Q \bigr\} < M_*,
\end{align}
the set defined in \eqref{D:Q*} is $L^2$-equicontinuous.
\end{definition}

Evidently, Theorem~\ref{T:equi} shows that $M_*\geq 4\pi$ and we conjecture that $M_*=\infty$.  Our primary contribution to the well-posedness problem is low-regularity well-posedness below the $M_*$ threshold:

\begin{theorem}\label{T:well}
Fix $\frac16\leq s<\frac12$.  The \eqref{DNLS} evolution is globally well-posed, both on the line and on the circle, in the space
\begin{align}
B^s_{M_*} = \bigl\{ q\in H^s : \| q\|_{L^2}^2 < M_* \bigr\} 
\end{align}
endowed with the $H^s$ topology.
\end{theorem}

A natural prerequisite for proving this theorem is a priori $H^s$ bounds.  In Section~\ref{sec; Hs}, we show how such bounds follow from $L^2$-equicontinuity; see Theorem~\ref{T:Hs conservation}. 

To prove Theorem~\ref{T:well} we employ the method of commuting flows introduced in \cite{KV18}.  In that paper, the method was used to prove well-posedness of the Korteweg--de Vries equation.  It has also been adapted and extended to treat the well-posedness problem for other completely integrable PDE \cite{bringmann2019global,HGKV}, to prove symplectic non-squeezing \cite{ntekoume2019symplectic}, and to construct dynamics for KdV in thermal equilibrium \cite{MR4145790}. 

In contrast to those papers, we do not employ a change of unknown; this simplifies some of the analysis.  On the other hand, new difficulties attend the construction of regularized flows: Because they are rooted in $\alpha(\kappa;q)$, the regularized Hamiltonians $H_\kappa(q)$ cannot be defined throughout $B^0_{M_*}$ for any single value of $\kappa$.  Instead, we need to use an exhaustion by equicontinuous subsets. Ultimately, these problems originate in the $L^2$-criticality of the problem.  Nevertheless, we will be able to prove that the regularized flows admit a satisfactory notion of well-posedness all the way down to $L^2$!  The $s\geq\frac16$ restriction arises later when we show that the regularized flows converge to the full \eqref{DNLS} evolution.

At this moment we do not know whether $s=\frac16$ is sharp in either geometry or indeed, whether the threshold regularity will differ between the line and the circle.  Moreover, we do not know of any results (in either geometry) that would preclude well-posedness all the way down to the scaling critical space $L^2$.  On the other hand, the self-similar solutions constructed in \cite{MR4079021} (see also \cite{MR818186}) show that smooth solutions can break-down in a dramatic way if one permits mere weak-$L^2$ decay at spatial infinity.

The restriction $s<\tfrac12$ in Theorem~\ref{T:well} does not represent a meaningful breakdown of our methods.  However, treating larger values would require additional arguments.  This seems unwarranted given that a great deal is already known about $H^s$-solutions for $s\geq \frac12$, as we shall now discuss.

Local well-posedness in $H^s$ for $s>\frac32$ was proved by Tsutsumi and Fukuda \cite{MR621533,MR634894}.  This was extended to $s\geq \frac12$ by Takaoka \cite{MR1693278} for \eqref{DNLS} posed on the line and by Herr \cite{MR2219223} for the periodic problem.  The endpoint $s=\frac12$ is significant: for lesser $s$, the data-to-solution map can no longer be uniformly continuous on bounded sets; see \cite{MR1837253, MR1693278}.

Global well-posedness in $H^1(\R)$ for initial data satisfying $M(q)<2\pi$ was obtained by Hayashi and Ozawa \cite{MR1152001}.  This result was extended first to $s>\frac23$ and then to $s>\frac12$ by Colliander, Keel, Staffilani, Takaoka, and Tao \cite{MR1871414, MR1950826}, under the same $L^2$ restriction.  See \cite{MR2823664} for a refinement of these arguments to handle the endpoint case $s=\frac12$, as well as \cite{MR1836810} for earlier efforts in this direction.

Hayashi and Ozawa \cite{MR1152001} also proved that solutions with initial data in $\Schw$ remain in $\Schw$ for as long as they remain bounded in $H^1$.

In \cite{MR3393674}, Wu proved global well-posedness in $H^1(\R)$ for initial data satisfying $M(q)<4\pi$; see also his earlier work \cite{MR3198590} which first overcame the $2\pi$ barrier. An alternate variational proof was given in \cite{MR3668587}, which also constructed global solutions for highly modulated initial data of arbitrary $L^2$ size.  The result in \cite{MR3393674} was extended to the periodic setting in \cite{MR3377682}.  Finally, the argument in \cite{MR1950826} was further advanced in \cite{MR3583477, MR3680936} to treat the endpoint case $s=\frac12$ and $M(q)<4\pi$; see also \cite{MR2668514} for earlier work in the periodic setting.

We note that the results of this paper provide an alternate proof of the main results in \cite{MR3377682,MR3393674}; see Corollary~\ref{C:H1}.  In particular, Proposition~\ref{P:H1 via H2} shows that $H^1$ bounds follow from Theorem~\ref{T:equi}.

The well-posedness of \eqref{DNLS} has also been investigated in Fourier-Lebesgue spaces; \cite{Deng2, MR2181058, MR2390318}.  This allowed the authors to obtain a uniformly continuous data-to-solution map in spaces that are closer to the critical scaling; recall that this property breaks down in $H^s$ spaces when $s<\frac12$.  An almost sure global well-posedness result for randomized initial data was proved in \cite{MR2928851}.

As a completely integrable PDE, \eqref{DNLS} is also amenable to inverse scattering techniques.
Building on the pioneering work of Liu \cite{MR3706093}, global well-posedness and asymptotic analysis of soliton-free solutions in $H^{2,2}(\R)= \{f\in H^2(\R):\, x^2f\in L^2(\R)\}$ were addressed in \cite{MR3563476, MR3739932}.

Global well-posedness for \emph{all} $H^{2,2}(\R)$ initial data was proved by Jenkins, Liu, Perry, and Sulem in \cite{MR4149070}.  This work builds on the authors' prior successes in \cite{MR3913998}.  These authors also proved a soliton resolution result \cite{MR3858827} for generic data in $H^{2,2}(\R)$.  See also their excellent review article \cite{MR4042219}.

The inverse scattering approach was also applied by Pelinovsky and Shimabukuro \cite{MR3862117} to prove global well-posedness in $H^{1,1}(\R)\cap H^2(\R)$ for soliton-free solutions and then in joint work with Saalmann \cite{MR3702542} for data giving rise to finitely many solitons; see also \cite{saalmann2017global}.

Recent months have witnessed a surge of activity on the well-posedness problem for \eqref{DNLS}.  First among these is the paper \cite{klaus2020priori}, which showed a priori $H^s$ bounds, $0<s<\frac12$, for solutions with $M(q)$ small.  The smallness assumption allows them to guarantee that the series \eqref{alpha defn} converges rapidly for $\kappa$ large, and so the series be conflated with its first term.  The paper \cite{tang2020microscopic} presents a microscopic representation of the conservation of $\alpha(\kappa;q)$.  In \cite{bahouri2020global}, Bahouri and Perelman achieve the major breakthrough of proving that for \emph{every} initial datum in $H^{1/2}(\R)$, the orbit remains bounded in the same space (irrespective of the size of $M(q)$).  For the periodic \eqref{DNLS}, the paper \cite{isom2020growth} shows that for $s\geq 1$ and $M(q)$ small, the $H^s(\T)$ norm of solutions grows at most polynomially in time.

While these exciting results appeared too recently to affect what we do in this paper,  their novelty and insightfulness give us every hope that the conjectures presented herein may soon be resolved.

\subsection*{Acknowledgements} R. K. was supported by NSF grant DMS-1856755 and M.~V. by grant DMS-1763074.

\section{Preliminaries} \label{sec;preliminaries}

Our conventions for the Fourier transform are
\begin{align*}
\hat f(\xi) = \tfrac{1}{\sqrt{2\pi}} \int_\R e^{-i\xi x} f(x)\,dx  \qtq{so} f(x) = \tfrac{1}{\sqrt{2\pi}} \int_\R e^{i\xi x} \hat f(\xi)\,d\xi
\end{align*}
for functions on the line and
\begin{align*}
\hat f(\xi) = \int_0^1 e^{- i\xi x} f(x)\,dx \qtq{so} f(x) = \sum_{\xi\in 2\pi\Z} \hat f(\xi) e^{i\xi x}
\end{align*}
for functions on the torus $\T$. These definitions of the Fourier transform are unitary on $L^2$ and yield the Plancherel identities
\begin{align*}
    \|f\|_{L^2(\mathbb R)}=\|\hat f\|_{L^2(\mathbb R)} \qtq{and}   \|f\|_{L^2(\mathbb T)}=\sum_{\xi\in 2\pi \mathbb Z}|\hat f(\xi)|^2,
\end{align*}
as well as the following convolution identity on $\R$:
\begin{align*}
    \widehat {fg}= \tfrac{1}{\sqrt{2\pi}} \hat f \ast \hat g.
\end{align*}

We use the standard Littlewood--Paley decomposition of a function,
$$
q = \sum_{N\in 2^\N} q_N,
$$
based on a smooth partition of unity on the Fourier side.  Here $q_1$ denotes the projection onto frequencies $|\xi|\leq 1$; for $N\geq 2$, $q_N$ contains frequencies~$|\xi|\sim N$.

The fact that the operators $\Lambda$ and $\Gamma$ defined in \eqref{Lambda} are Hilbert--Schmidt was noticed already in \cite[Lemma~4.1]{KVZ17}:

\begin{lemma}\label{L:HS}
For $q\in L^2$ and $\kappa>0$ we have 
\begin{align}
    \|\Lambda\|_{\hs(\R)}^2=\|\Gamma\|_{\hs(\R)}^2&\approx  \int_{\mathbb R} \log(4+\tfrac{\xi^2}{\kappa^2}) \frac{|\hat q(\xi)|^2}{\sqrt{4\kappa^2 +\xi^2}}\,d\xi\lesssim \kappa^{-1}\|q\|_{L^2}^2,\label{HS real}\\
    \|\Lambda\|_{\hs(\T)}^2= \|\Gamma\|_{\hs(\T)}^2&\approx  \sum_{\xi\in 2\pi\mathbb Z} \log(4+\tfrac{\xi^2}{\kappa^2}) \frac{|\hat q(\xi)|^2}{\sqrt{4\kappa^2 +\xi^2}}\lesssim \kappa^{-1}\|q\|_{L^2}^2.\label{HS torus}
\end{align}
\end{lemma}

\begin{proof}
The estimate \eqref{HS real} follows from the computation
\begin{align*}
    \|\Lambda\|_{\hs(\R)}^2&= \tfrac{1}{2\pi} \int_{\mathbb R} |\hat q(\xi)|^2 \int_{\mathbb R} \tfrac{1}{\sqrt{\kappa^2 +\eta^2}\sqrt{\kappa^2 +(\eta+\xi)^2}}\,d\eta\,d\xi
    \approx  \int_{\mathbb R} \log(4+\tfrac{\xi^2}{\kappa^2}) \tfrac{|\hat q(\xi)|^2}{\sqrt{4\kappa^2 +\xi^2}}\,d\xi.
\end{align*}
To compute the above integral in $\eta$, one treats separately the regions $|\eta|\leq 2|\xi|$ and $|\eta|>2|\xi|$; the logarithm term arises only when considering the first region.

On the torus, similar arguments yield
\begin{align*}
    \|\Lambda\|_{\hs(\T)}^2&= \sum_{\xi\in 2\pi\mathbb Z} |\hat q(\xi)|^2 \sum_{\eta\in 2\pi\mathbb Z} \tfrac{1}{\sqrt{\kappa^2 +\eta^2}\sqrt{\kappa^2 +(\eta+\xi)^2}}\approx  \sum_{\xi\in 2\pi\mathbb Z} \log(4+\tfrac{\xi^2}{\kappa^2}) \tfrac{|\hat q(\xi)|^2}{\sqrt{4\kappa^2 +\xi^2}},
\end{align*}
which settles \eqref{HS torus}.
\end{proof}

These Hilbert--Schmidt bounds ensure that $i\kappa\Lambda\Gamma$ is trace class and thus that the determinant in \eqref{a defn} is well defined.  The trace of this operator will also be important and is easily evaluated:

\begin{lemma}\label{L:dominant}
Let $q\in L^2$ and $\kappa>0$. Then
\begin{alignat}{2}
 \tr(i\kappa \Lambda \Gamma)&= \int \tfrac{i\kappa |\hat q(\xi)|^2}{2\kappa-i\xi} \,d\xi \qquad&&\text{on $\R$},\label{trace real}\\
 \tr(i\kappa \Lambda \Gamma)&= \tfrac{1+e^{-\kappa}}{1-e^{-\kappa}} \!\sum_{\xi\in 2\pi\mathbb Z}\! \tfrac{i\kappa |\hat q(\xi)|^2}{2\kappa - i\xi}\qquad&&\text{on $\T$}\label{trace periodic}.
\end{alignat}
\end{lemma}  

\begin{proof}
To prove \eqref{trace real}, we simply compute the trace on the Fourier side:
\begin{align*}
\tr(i\kappa \Lambda \Gamma)&= \tfrac{i\kappa}{2\pi} \iint \tfrac{|\hat q(\xi)|^2}{(\eta-i\kappa)(\eta+\xi+i\kappa)}\,d\eta\,d\xi
	= \int \tfrac{i\kappa |\hat q(\xi)|^2}{2\kappa-i\xi} \,d\xi .
\end{align*}

In the circle setting, we use the partial fraction decomposition of the cotangent:
\begin{align}\label{pfrac for cot}
    \sum_{\eta\in 2\pi \mathbb Z}\bigl( \tfrac{1}{\kappa+i\eta} + \tfrac{1}{\kappa-i\eta} \bigr)= \coth (\tfrac{\kappa}{2})= \tfrac{1+e^{-\kappa}}{1-e^{-\kappa}}.
\end{align}
In this way, we find
\begin{align*}
    \tr(i\kappa \Lambda \Gamma)&= i \kappa \sum_{\xi\in 2\pi\mathbb Z} |\hat q(\xi)|^2 \tfrac{1}{\xi +2i\kappa} \sum_{\eta\in 2\pi\mathbb Z} \bigl( \tfrac{1}{\eta-i\kappa}- \tfrac{1}{\eta+\xi+i\kappa} \bigr)
    	=  \sum_{\xi\in 2\pi\mathbb Z}  \! \tfrac{i\kappa |\hat q(\xi)|^2}{2\kappa-i\xi}   \tfrac{1+e^{-\kappa}}{1-e^{-\kappa}}.
\end{align*}
Notice that the sum over $\eta$ simplifies to \eqref{pfrac for cot} because $\xi\in 2\pi\Z$.
\end{proof}

In Section~\ref{sec; gwp}, it will be convenient to express the next term in the series \eqref{alpha defn} as a paraproduct.  This is the role of the next lemma.

\begin{lemma}\label{L:sub dominant}
Let $q\in L^2$ and $\kappa>0$. Then
\begin{alignat}{2}\label{sub dom R}
 \tr\bigl( [ \Lambda \Gamma]^2\bigr)&= \int_\R \left(\tfrac{1}{2\kappa+\partial} \bar q \right)^2(4\kappa-\partial)\left(\tfrac{1}{2\kappa-\partial} q\right)^2 \,dx \qquad&&\text{on $\R$},\\
\label{sub dom T}
 \tr\bigl( [\Lambda \Gamma]^2\bigr)&=  \tfrac{1+e^{-\kappa}}{1-e^{-\kappa}}
 	\int_\T \left(\tfrac{1}{2\kappa+\partial} \bar q \right)^2(4\kappa-\partial)\left(\tfrac{1}{2\kappa-\partial} q\right)^2 \,dx\qquad&&\text{on $\T$}.
\end{alignat}
\end{lemma}

\begin{proof}
The method is exactly that of the previous lemma, only the details change.  In the line case, we have a more complicated (but still elementary) contour integral.  In the circle case, one must verify that
\begin{align*}
\sum_{\xi\in2\pi\Z} &\frac{1}{(\kappa+i\xi)(\kappa-i[\xi+\eta_1])(\kappa+i[\xi+\eta_1+\eta_2])(\kappa-i[\xi+\eta_1+\eta_2+\eta_3])}  \\
&=\frac{1+e^{-\kappa}}{1-e^{-\kappa}} \cdot \frac{4\kappa-i(\eta_1+\eta_3)}{(2\kappa-i\eta_1)(2\kappa+i\eta_2)(2\kappa-i\eta_3)(2\kappa+i\eta_4)}.
\end{align*}
This follows from \eqref{pfrac for cot} via a careful partial fraction decompostion.
\end{proof}

Our next lemma records operator estimates for frequency localized potentials.

\begin{lemma}[Operator estimates]\label{L:op est}
Fix $q\in L^2$, $N\in 2^\N$, and $\kappa\geq 1$, and denote $\Lambda_N=\Lambda(q_N)$ and $\Gamma_N=\Gamma(q_N)$. Then
\begin{align}
    &\|\Lambda_N\|_{\hs}= \|\Gamma_N\|_{\hs} \approx \sqrt{\tfrac{1}{\kappa +N} \log\left(4+\tfrac{N^2}{\kappa^2}\right)} \|q_N\|_{L^2}, \label{hilbert_schmidt}\\
    &\|\Lambda_N\|_{\op}= \|\Gamma_N\|_{\op} \lesssim \min\left\{\tfrac{\sqrt{N}}{\kappa} , \sqrt{\tfrac{1}{\kappa +N} \log\left(4+\tfrac{N^2}{\kappa^2}\right)}\right\}  \|q_N\|_{L^2},\label{operator}\\
    &\sum_{N\leq N_0}\|\Lambda_N\|_{\op} \lesssim \kappa^{-1} \min\left\{\sqrt{N_0}, \sqrt{\kappa}\right\}\|q\|_{L^2}.\label{operator_sum}
\end{align}
\end{lemma}

\begin{proof}
The claim \eqref{hilbert_schmidt} follows immediately from Lemma~\ref{L:HS}.  

Using the Bernstein inequality, we estimate
\begin{align*}
\|\Lambda_N\|_{\op}\leq \|(\kappa-\partial)^{-\frac 12}\|_{\op} \|q_N\|_{\op}\|(\kappa+\partial)^{-\frac 12}\|_{\op}
\leq \tfrac{1}{\kappa} \|q_N\|_{L^\infty}\lesssim \tfrac{\sqrt{N}}{\kappa} \|q_N\|_{L^2}.
\end{align*}
Combining this with \eqref{hilbert_schmidt} yields \eqref{operator}.

The case $N_0\leq \kappa$ of \eqref{operator_sum} is clear. If $N_0>\kappa$, an application of \eqref{operator} yields
\begin{align*}
    \sum_{N\leq N_0}\|\Lambda_N\|_{\op}&\lesssim \sum_{N\leq \kappa} \tfrac{\sqrt{N}}{\kappa} \|q\|_{L^2} +\sum_{\kappa<N\leq N_0}\sqrt{\tfrac{1}{N} \log\left(4+\tfrac{N^2}{\kappa^2}\right)} \|q\|_{L^2} \lesssim \tfrac{\sqrt{\kappa}}{\kappa} \|q\|_{L^2},
\end{align*}
as desired.
\end{proof}

\begin{lemma} \label{L:H^s estimates}
For all $\kappa\geq 1$, we have
\begin{align}
    \|(\kappa+\partial)^{-1} f (\kappa-\partial)^{-1}\|_{\hs}&\lesssim \kappa^{-\frac 12}\|f\|_{H^{-1}}, \label{estimate 1}\\
    \|q (\kappa+\partial)^{-\frac34}\|_{\hs}&\lesssim \kappa^{-\frac 14}\|q\|_{L^2}, \label{estimate 2}\\
    \|(\kappa-\partial)^{-\frac14} q (\kappa+\partial)^{-\frac14}\|_{\op}&\lesssim\|q\|_{L^2}. \label{estimate 3}
\end{align}
\end{lemma}

\begin{proof}
We first turn to \eqref{estimate 1}. We will only consider here the line setting; in the periodic case, one can apply a similar argument to the one in the proof of Lemma~\ref{L:HS}. A straightforward computation yields
\begin{align*}
    \|(\kappa+\partial)^{-1} f (\kappa-\partial)^{-1}\|_{\hs}^2&= \tfrac{1}{2\pi}\iint \tfrac{|\hat f(\xi)|^2}{[\kappa^2+(\xi+\eta)^2] (\kappa^2+\eta^2)}\,d\eta \,d\xi.
\end{align*}
Considering separately the regions $|\eta|\leq 2|\xi|$ and $|\eta|>2|\xi|$ when integrating in $\eta$, we find
\begin{align*}
    \|(\kappa+\partial)^{-1} f (\kappa-\partial)^{-1}\|_{\hs}^2\lesssim \int\tfrac{ |\hat f(\xi)|^2 }{\kappa(\kappa^2+\xi^2)}\,d\xi\lesssim \kappa^{-1}\|f\|_{H^{-1}}^2.
\end{align*}

By direct computation (cf. \cite[Theorem~4.1]{MR2154153}), we have
$$
\|q (\kappa+\partial)^{-\frac34}\|_{\hs}\lesssim \|q\|_{L^2}\|(\kappa+i\xi)^{-\frac34}\|_{L^2_\xi}\lesssim  \kappa^{-\frac 14}\|q\|_{L^2},
$$
which settles \eqref{estimate 2}.

Similarly, by Cwikel's theorem (see \cite{MR473576} or \cite[Theorem~4.2]{MR2154153}), we find that 
\begin{align*}
\|(\kappa-\partial)^{-\frac14} q (\kappa+\partial)^{-\frac14}\|_{\op}&\leq \bigl\|(\kappa-\partial)^{-\frac14} \sqrt{|q|}\bigr\|_{\op}\bigl\|\tfrac{q}{\sqrt{|q|}} (\kappa+\partial)^{-\frac14}\bigr\|_{\op}\\
&\lesssim \bigl\|(\kappa\pm i\xi)^{-\frac14}\bigr\|_{L^4_{\text{weak}}}^2 \bigl\| \sqrt{|q|}\bigr\|_{L^4}^2\lesssim \|q\|_{L^2}. \qedhere
\end{align*}
\end{proof}

\begin{proposition} \label{P:large kappa}
Let $Q$ be a bounded and equicontinuous subset of $L^2$. Then
\begin{align}\label{equi'}
\lim_{\kappa\to \infty} \sup_{q\in Q} \sqrt{\kappa} \|\Lambda(q)\|_{\op}=0.
\end{align}
Moreover, there exists $\kappa_0\geq 1$ so that the series \eqref{alpha defn} converges uniformly for $\kappa\geq \kappa_0$ and $q\in Q$.
\end{proposition}

\begin{proof}
Fix $\varepsilon>0$ and let $\eta>0$ be a small parameter to be chosen later. Using \eqref{operator_sum} and Lemma~\ref{L:HS}, we get
\begin{align*}
    \sqrt{\kappa} \bigl\|\Lambda(q)\bigr\|_{\op}&
 \lesssim  \sqrt{\kappa} \, \bigl\|\Lambda(q_{>\eta \kappa})\bigr\|_{\op} + \sqrt{\kappa} \sum_{N\leq \eta\kappa} \|\Lambda_N(q)\|_{\op}
 \lesssim \|q_{>\eta \kappa} \|_{L^2} + \sqrt{\eta} \, \|q\|_{L^2}.
\end{align*}
Choosing $\eta$ small enough depending on the $L^2$ bound of $Q$, and then $\kappa$ sufficiently large depending on $\eta$ and the equicontinuity property of $Q$, we may ensure that
\begin{align*}
     \sqrt{\kappa} \, \|\Lambda(q)\|_{\op}<\varepsilon\quad\text{for all}\,\, q\in Q,
\end{align*}
which yields \eqref{equi'}.

To continue, we choose $\kappa_0$ sufficiently large so that for any $\kappa\geq \kappa_0$ we have $\sqrt{\kappa}\,\|\Lambda(q)\|_{\op}\leq\frac12$ uniformly for $q\in Q$.  Lemma~\ref{L:HS} then yields
\begin{align}\label{geometric decay}
\bigl\| ( i \kappa \Lambda \Gamma)^{\ell+1}\bigr\|_{\tc}  \leq \kappa^{\ell+1} \|\Lambda\|_{\hs}^2 \|\Lambda\|_{\op}^{2\ell}
	\lesssim 2^{-\ell} \|q\|_{L^2}^2 ,
\end{align}
uniformly for $\kappa\geq \kappa_0$ and $q\in Q$, which ensures convergence of the series \eqref{alpha defn}.
\end{proof}

As discussed in the introduction, this convergence result allows the arguments of \cite{klaus2020priori,tang2020microscopic} to be extended beyond the regime of small $L^2$ norm and so show that $\alpha(\kappa;q)$ is conserved under the \eqref{DNLS} flow, for $\kappa$ sufficiently large.  This conservation is inherited by $a(\kappa;q)$ for \emph{all} $\Re \kappa>0$ because this is a holomorphic function in this region.

\section{Equicontinuity in $L^2$}\label{sec;equicontinuity}

The goal of this section is to prove Theorem~\ref{T:equi}. We begin with a convenient notion of the momentum at high frequencies in each geometry:
\begin{align}\label{beta2 defn}
    \beta_{\R}^{[2]}(\kappa;q):=\int_{\mathbb R} \frac{\xi^2 |\hat q(\xi)|^2 }{4\kappa^2+\xi^2 }d\xi 
\qtq{and}
    \beta_{\T}^{[2]}(\kappa;q):= \sum_{\xi\in 2\pi\mathbb Z} \frac{\xi^2|\hat q(\xi)|^2}{4\kappa^2+\xi^2 }.
\end{align}
The curious notation is explained by the fact that these expressions coincide with the quadratic (in $q$) parts of the quantities in \eqref{beta defn}.  For our immediate purposes, however, the following relation with the formulas of Lemma~\ref{L:dominant} is more important:
\begin{equation}\label{beta2 vs tr}
\begin{alignedat}{2}
    \Im \tr(i\kappa \Lambda \Gamma) &= \tfrac12\bigl[M(q) - \beta_{\R}^{[2]} (\kappa;q)\bigr] &&\text{on}\,\,\mathbb R,\\
    \Im \tr(i\kappa \Lambda \Gamma) &= \tfrac12\,\tfrac{1+e^{-\kappa}}{1-e^{-\kappa}}\bigl[M(q) - \beta_{\T}^{[2]} (\kappa;q)\bigr]& \qquad &\text{on}\,\,\mathbb T.
\end{alignedat}
\end{equation}

Given an infinite subset $\mathcal K\subseteq 2^{\mathbb N}$, we then define a norm via
\begin{align}\label{K defn}
    \|q\|_{\mathcal K}^2&:=\|q\|_{L^2}^2 + \sum_{\kappa\in \mathcal K} \beta^{[2]}(\kappa;q).
\end{align}
This in turn leads to a very convenient formulation of equicontinuity:

\begin{lemma}\label{L:equi via K}
A set $Q\subseteq L^2$ is bounded and equicontinuous if and only if there exists an infinite set $\mathcal K\subseteq 2^{\mathbb N}$ so that $\sup_{q\in Q}\|q\|_{\mathcal K}<\infty$. 
\end{lemma}

\begin{proof}
This is immediately evident from the observation that 
\begin{align*}
    \|q\|_{\mathcal K}^2 \approx \|q\|_{L^2}^2 + \sum_{\kappa\in \mathcal K} \|q_{>\kappa}\|_{L^2}^2
    	&\approx \|q\|_{L^2}^2 + \sum_{N\in 2^{\mathbb N}} \#\{\kappa\in \mathcal K: \kappa< N\}\;\|q_N\|_{L^2}^2. \qedhere
\end{align*}
\end{proof}

Before beginning the proof of Theorem~\ref{T:equi}, we need two further preliminaries.  The first will allow us to pass from the determinant to the exponentiated trace, and  the second to take logarithms.

\begin{lemma}\label{det 2 ish}
Let $A\in \tc$.  Then
\begin{align}\label{E:det 2 ish}
\bigl| \det( 1 + A ) - \exp\{\tr(A)\} \bigr| \leq \tfrac12 \| A \|_{\hs}^2 \exp\bigl\{ \| A\|_{\tc} \bigr\}.
\end{align}
\end{lemma}

\begin{proof}
Let $\lambda_i$ enumerate the non-zero eigenvalues of $A$ repeated according to algebraic multiplicity.  By relating eigenvalues and singular values, Weyl proved that
\begin{align*}
\sum |\lambda_i| \leq \| A \|_{\tc} \qtq{and} \sum |\lambda_i|^2 \leq \| A \|_{\hs}^2  .
\end{align*}

Now let us compare
\begin{align*}
\det( 1 + A ) &= 1 + \sum_{n=1}^\infty \; \frac{1}{n!} \sum_{\substack{i_1,\ldots,i_n\\\text{distinct}}} \lambda_{i_1}\lambda_{i_2}\cdots \lambda_{i_n} ,\\
\exp\{ \tr(A) \} &= 1 + \sum_{n=1}^\infty \; \frac{1}{n!} \sum_{i_1,\ldots,i_n} \lambda_{i_1}\lambda_{i_2}\cdots \lambda_{i_n}.
\end{align*}
Evidently, the difference contains only sums over $n$-tuples $(i_1,\ldots,i_n)$ that contain at least one pair of identical indices.  Thus,
\begin{align*}
\text{LHS\eqref{E:det 2 ish}} &\leq \sum_{n=2}^\infty \frac{1}{n!}  \binom{n}{2} \Bigl[ \sum_j |\lambda_j|^2\Bigr] \Bigl[ \sum_i |\lambda_i| \Bigr]^{n-2}
	\leq \tfrac12 \| A \|_{\hs}^2 \sum_{n=2}^\infty \frac{1}{(n-2)!} \|A\|_{\tc}^{n-2}
\end{align*}
and so \eqref{E:det 2 ish} follows.
\end{proof}

\begin{lemma}\label{L:unwrap}
Given $C>0$ and $0<\eps<\pi$, let
\begin{align}\label{A rectangle}
\mathcal R=\{z:|\Re z|\leq C \text{ and } 0<\Im z<2\pi-\eps\}.
\end{align}
Then 
\begin{align}\label{invert exp}
 \bigl|\Im(z-w)\big| \leq \frac{\pi e^{C}}{\sin(\eps/2)} \bigl| e^w - e^z \bigr| \qtq{uniformly for} z,w\in\mathcal R.
\end{align}
\end{lemma}

\begin{proof}
This reduces to elementary trigonometry once one realizes that the worst-case scenario is $\Re z = \Re w =-C$. 
\end{proof}

We are now ready for the climax of the section:

\begin{proof}[Proof of Theorem~\ref{T:equi}]  Let us begin right away with the key computation.  Given any $q\in L^2$, we may apply \eqref{hilbert_schmidt}, \eqref{operator_sum}, and (in the final step) Cauchy--Schwarz to deduce that
\begin{align}\label{sum I_2}
    \sum_{\kappa\in2^\N}& \|i\kappa\Lambda(q)\Gamma(q)\|_{\hs}^2\\
    &\lesssim \sum_{\kappa\in 2^\N}\kappa^2 \sum_{N_1\sim N_2\geq N_3, N_4} \|\Lambda_{N_1}(q)\|_{\hs} \|\Lambda_{N_2}(q)\|_{\hs}\|\Lambda_{N_3}(q)\|_{\op}\|\Lambda_{N_4}(q)\|_{\op} \notag \\
     &\lesssim M(q) \sum_{\kappa\in2^\N} \sum_{N_1\sim N_2}  \tfrac{1}{N_2+\kappa}\log\left(4+\tfrac{N_2^2}{\kappa^2}\right) \|q_{N_1}\|_{L^2} \|q_{N_2}\|_{L^2} \min\{N_2,\kappa\} \notag \\
     & \lesssim M(q) \sum_{N_1\sim N_2} \|q_{N_1} \|_{L^2} \|q_{N_2} \|_{L^2} \Bigl(\sum_{\kappa\leq N_2}   \tfrac{\kappa}{N_2}\log\left(4+\tfrac{N_2^2}{\kappa^2}\right)  + \sum_{\kappa>N_2}   \tfrac{N_2}{\kappa}\Bigr) \notag \\
     & \lesssim M(q)^2. \notag
\end{align}
Combining this with Lemmas~\ref{L:HS} and~\ref{det 2 ish}, we find
\begin{align}\label{det vs tr}
\sum_{\kappa\in2^\N} \bigl| a(\kappa;q) - \exp\bigl\{- \tr\bigl[i\kappa \Lambda(q)\Gamma(q)\bigr]\bigr\}\bigr| \leq C M(q)^2 e^{C M(q)}  
\end{align}
for some absolute $C$.

As we did not explicitly require that $M(\tilde q) = M(q)$ for $\tilde q\in C_q^0$, let us pause to see that this follows from the equality $a(\kappa; \tilde q)\equiv a(\kappa; q)$.  From \eqref{E:det 2 ish} and Lemma~\ref{L:dominant} we see that for $\kappa\to\infty$,
\begin{align*}
0= \bigl| a(\kappa;\tilde q) - a(\kappa;q)\bigr| 
&= \bigl| \exp\bigl\{ -\tr\bigl[i\kappa \Lambda\bigl(\tilde q\bigr)\Gamma\bigl(\tilde q\bigr)\bigr]
		- \exp\bigl\{- \tr\bigl[i\kappa \Lambda\bigl(q\bigr)\Gamma\bigl(q\bigr)\bigr]\bigr\}  \bigr|  + o(1)\\
&= \bigl| \exp\bigl\{ -\tfrac{i}2 M\bigl(\tilde q\bigr) \bigr\} - \exp\bigl\{-\tfrac{i}2 M\bigl(q\bigr) \bigr\}\bigr| + o(1).
\end{align*}
Thus $M(\tilde q)$ is preserved modulo $4\pi\Z$.  As $\tilde q$ belongs to the same connected component as $q$, we must have that $M(\tilde q) = M(q)$.  For later use, we note the consequence
\begin{align}\label{M equal}
\sup_{q\in Q_{**}} M(q) = \sup_{q\in Q} M(q).
\end{align}

While this argument did not require the hypothesis \eqref{mass bound}, we will need it to unwrap this phase ambiguity when we address equicontinuity.  This is our next topic.

Given an equicontinuous set $Q$ satisfying \eqref{mass bound}, let us choose $\eps>0$ and an infinite subset $\mathcal K\subseteq 2^\N$ so that
$$
\sup_{q\in Q,\kappa\in\mathcal K} \tfrac{1+e^{-\kappa}}{1-e^{-\kappa}} M(q)  \leq 4\pi - 2\eps \qtq{and} \sup_{q\in Q} \| q \|_{\mathcal K} <\infty.
$$
Proceeding very much as we did above, we see that
\begin{align*}
\sum_{\kappa\in\mathcal K} \bigl| \exp\bigl\{- \tr\bigl[i\kappa \Lambda\bigl(\tilde q\bigr)\Gamma\bigl(\tilde q\bigr)\bigr]
		- \exp\bigl\{- \tr\bigl[i\kappa \Lambda\bigl(q\bigr)\Gamma\bigl(q\bigr)\bigr]\bigr\}  \bigr|
\leq  2 C M(q)^2 e^{C M(q)}
\end{align*}
for any $\tilde q\in C_q^0$.  Combining this with \eqref{beta2 vs tr} and Lemma~\ref{L:unwrap}, we deduce that
\begin{align*}
\sum_{\kappa\in\mathcal K} \Bigl|  \beta^{[2]}\bigl(\kappa;\tilde q\bigr) - \beta^{[2]}\bigl(\kappa;q\bigr) \Bigr| \lesssim_\eps 1.
\end{align*}
This in turn guarantees that
$$
\sup\bigl\{ \|\tilde q \|_{\mathcal K}^2 : \tilde q\in Q_{**}\bigr\} \leq \sup\bigl\{ \|q \|_{\mathcal K}^2 :  q\in Q\bigr\} + O_\eps(1) < \infty,
$$
from which equicontinity follows via Lemma~\ref{L:equi via K}.
\end{proof}

\section{Conservation laws and Equicontinuity} \label{sec; Hs}

The primary goal of this section is to prove $H^s$ bounds for \eqref{DNLS} solutions, for $0<s<\frac12$, as a prerequisite for proving Theorem~\ref{T:well}.  We will also prove equicontinuity in these spaces, which is also needed to prove that theorem.

Before turning to that subject, we pause to show how $L^2$-equicontinuity can be used to restore coercivity to the traditional polynomial conservation laws.   As a representative example, we show how $H_2(q)$ can be used to control the $H^1$-norm: 

\begin{proposition}\label{P:H1 via H2}
Let $Q\subseteq H^1$ be $L^2$-bounded and equicontinuous.  Then
\begin{align}\label{H1}
\|q \|_{H^1}^2 \lesssim  H_2(q) +M(q)^3,
\end{align}
uniformly for all $q\in Q$.
\end{proposition}

\begin{proof}
Splitting into low and high frequency parts and estimating using the Bernstein and Gagliardo--Nirenberg inequalities, respectively, we obtain
\begin{align*}
    \|q\|_{L^6}^6&\lesssim \|q_{\leq N}\|_{L^6}^6 + \|q_{>N}\|_{L^6}^6
    \lesssim N^2 \|q\|_{L^2}^6 + \|q_{>N}\|_{L^2}^4 \|q'\|_{L^2}^2.
\end{align*}
This allows us to control the quartic term in $H_2$, and hence the $H^1$-norm, as follows:
\begin{align*}
    \|q'\|_{L^2}^2&\leq H_2(q) + \tfrac{3}{2} \int |q(x)|^3 |q'(x)|\,dx\\
    & \leq H_2(q)+  \varepsilon \|q'\|_{L^2}^2+ \tfrac{9}{16 \varepsilon}\|q\|_{L^6}^6\\
    & \leq H_2(q)+  \left(\varepsilon +C\tfrac{9}{16 \varepsilon} \|q_{> N}\|_{L^2}^4 \right) \|q'\|_{L^2}^2+ C\tfrac{9}{16 \varepsilon} N^2 M(q)^3,
\end{align*}
for any $\varepsilon>0$.  The claim \eqref{H1} now follows by choosing $\varepsilon$ small and then $N$ large, exploiting the equicontinuity of $Q$.
\end{proof}

Proposition~\ref{P:H1 via H2} allows us to extend local $H^1$ solutions globally in time, provided we remain below the $M_*$ bound introduced in Definition~\ref{D:M_*}.

\begin{corollary}\label{C:H1}
The \eqref{DNLS} evolution is globally well-posed, both on the line and on the circle, in the space
\begin{align}
B^1_{M_*} = \bigl\{ q\in H^1 : \| q\|_{L^2}^2 < M_* \bigr\} 
\end{align}
endowed with the $H^1$ topology.  Moreover, initial data in $\Schw$ leads to solutions that belong to $\Schw$ at all times.
\end{corollary}

\begin{proof}
In the line case, this result can be deduced from \cite{bahouri2020global}; indeed, the restriction $M(q)<M_*$ is not needed in this case.  Below we give an alternate argument that works also in the periodic setting.

As discussed in the introduction, local well-posedness in $H^1$ was proved already in \cite{MR1693278,MR2219223}.  Thus, given initial data $q(0)\in B^1_{M_*}\cap \Schw$, there is a corresponding maximal lifespan solution $q\in C_t([0,T); H^1)$ to \eqref{DNLS}.  Moreover, \cite{MR1152001} shows that $q(t)\in\Schw$ for all $t\in [0,T)$.  Combining \cite{klaus2020priori,tang2020microscopic}  with Proposition~\ref{P:large kappa} yields that $a(\kappa; q(t)) =a(\kappa, q(0))$ for all $t\in [0,T)$ and $\kappa>0$. By the definition of $M_*$ and Proposition~\ref{P:H1 via H2}, the solution $q$ satisfies a priori $H^1$ bounds on $[0,T)$, which in turn guarantees that $T=\infty$.

Finally, global well-posedness in $B^1_{M_*}$ follows from local well-posedness and the density of $\Schw$ in $H^1$.
\end{proof}

Let us now turn to low-regularity questions.  Bounded sets in $H^s$, $s>0$, are automatically bounded and equicontinuous in $L^2$.   As we shall work only below the $M_*$ threshold in this section, such $L^2$-equicontinuity is retained globally in time. Our goal is to propagate $H^s$ bounds.  The key to doing this is a certain renormalization of $\alpha(\kappa;q)$ that we introduce now:
\begin{equation}\label{beta defn}
\begin{alignedat}{2}
    \beta_{\mathbb R}(\kappa;q)&:= \|q\|_{L^2}^2-2 \Im \alpha(\kappa;q) &&\text{on}\,\,\mathbb R,\\
    \beta_{\mathbb T}(\kappa;q)&:= \|q\|_{L^2}^2- \tfrac{1-e^{-\kappa}}{1+e^{-\kappa}}\, 2 \Im \alpha(\kappa;q) &\qquad &\text{on}\,\,\mathbb T.
\end{alignedat}
\end{equation}
Proposition~\ref{P:large kappa} guarantees that these quantities are well defined for $\kappa$ sufficiently large across our whole family of orbits.

The quadratic (in $q$) parts of these expressions were presented already in \eqref{beta2 defn}.  As we saw there, these provide a sense of the $L^2$-norm of the high-frequency part of $q$.  To address higher regularity, for $0<s<\frac12$ we consider the quantity
\begin{align*}
    \beta_{s}(\kappa;q):= \int_\kappa^\infty \beta(\varkappa;q) \varkappa^{2s} \tfrac{d\varkappa}{\varkappa}.
\end{align*}
The quadratic term in this expression is given by
\begin{align*}
    \beta_{s}^{[2]}(\kappa;q)&= \int_\kappa^\infty \beta^{[2]}(\varkappa;q)\, \varkappa^{2s}\, \tfrac{d\varkappa}{\varkappa}=\int_\kappa^\infty \bigl\langle \tfrac{-\partial^2}{4\varkappa^2-\partial^2}q,q \bigr\rangle\,  \varkappa^{2s} \, \tfrac{d\varkappa}{\varkappa}\approx_s \bigl\langle \tfrac{-\partial^2}{(\kappa^2-\partial^2)^{1-s}}q,q \bigr\rangle.
\end{align*}
From this we see that for any $0<\eta<1$,
\begin{align}\label{Hs high}
      \|q_{>\kappa}\|_{H^s}^2 \lesssim \beta_{s}^{[2]}(\kappa;q) \lesssim \eta^{2(1-s)} \|q\|_{H^s}^2 + \|q_{>\eta\kappa}\|_{H^s}^2,
\end{align}
and so $\beta_{s}^{[2]}(\kappa;q)$ captures the $H^s$-norm of the high-frequency part of $q$.  Indeed, a bounded set $Q\subseteq H^s$ is equicontinuous in $H^s$ if and only if $\beta_{s}^{[2]}(\kappa;q)\to 0$ uniformly on $Q$ as $\kappa\to\infty$. 

\begin{theorem} \label{T:Hs conservation}
Fix $0<s<\tfrac{1}{2}$ and let $\,Q\subseteq \mathcal S$ be $H^s$-bounded and satisfy \eqref{M_*}.  Then, recalling the notation $\,Q_{**}$ from \eqref{D:Q*}, we have
\begin{align} \label{Hs bound}
    \sup_{q\in Q_{**}}\|q\|_{H^s} \lesssim C\Bigl( \, \sup_{q\in Q} \| q\|_{L^2}^2 , \, \sup_{q\in Q} \| q\|_{H^s}^2 \Bigr).
\end{align}
Moreover, if $Q$ is $H^s$-equicontinuous, then so is $Q_{**}$.
\end{theorem}

\begin{proof}
As $Q$ is $H^s$-bounded,  it is automatically $L^2$-bounded and equicontinuous.  By \eqref{M equal}, $Q_{**}$ inherits $L^2$-boundedness from $Q$. As $Q$ satisfies \eqref{M_*}, we deduce that $Q_{**}$ is also $L^2$-equicontinuous. By Proposition~\ref{P:large kappa}, we may choose $\kappa_0\geq 1$ so that
\begin{align}\label{little op again}
\sqrt{\kappa}\,\bigl\|\Lambda(q)\bigr\|_{\op}\leq\tfrac12 \qquad\text{uniformly for $ q\in Q_{**}$ and $\kappa\geq \kappa_0$.}
\end{align}
As shown there, this ensures that $\alpha(\kappa;q)$ and so also $\beta(\kappa;q)$ are well defined for all $q\in Q_{**}$ and $\kappa\geq \kappa_0$.

Arguing as in \eqref{geometric decay}, we also see that \eqref{little op again} implies
\begin{align} 
    | \beta_s(\kappa;q)&- \beta^{[2]}_s(\kappa;q)|\lesssim \int_\kappa^\infty  \varkappa^{2s+2} \|\Lambda(q) \Gamma(q)\|_{\hs}^2 \tfrac{d\varkappa}{\varkappa} \notag\\
    &\lesssim \int_\kappa^\infty  \varkappa^{2s+2} \sum_{N_1\sim N_2\geq N_3\geq N_4} \|\Lambda_{N_1}\|_{\hs} \|\Lambda_{N_2}\|_{\hs}\|\Lambda_{N_3}\|_{\op}\|\Lambda_{N_4}\|_{\op} \tfrac{d\varkappa}{\varkappa}\label{beta_s mess}
\end{align}
uniformly for $q\in Q_{**}$ and $\kappa\geq \kappa_0$.  To continue from here, we decompose the full sum into the subregions $S_j$ defined by
\begin{align*}
    S_1&=\{N_2\leq \kappa\},\\
    S_2&=\{\kappa<N_2\leq \varkappa \text{ and } N_3\leq \eta \kappa\},\\
    S_3&=\{\kappa<N_2\leq \varkappa \text{ and } N_3> \eta \kappa\},\\
    S_4&=\{N_2> \varkappa \text{ and } N_3\leq \eta \kappa\},\\
    S_5&=\{N_2> \varkappa \text{ and } N_3> \eta \kappa\},
\end{align*}
where $\eta\in(0,1)$ is a small parameter to be chosen later.  We will estimate separately each of the contributions
$$
I_j(\kappa;q) := \int_\kappa^\infty  \varkappa^{2s+2} \smash{\sum_{S_j}}\, \|\Lambda_{N_1}\|_{\hs} \|\Lambda_{N_2}\|_{\hs}\|\Lambda_{N_3}\|_{\op}\|\Lambda_{N_4}\|_{\op} \tfrac{d\varkappa}{\varkappa}.
$$

Applying \eqref{hilbert_schmidt} and \eqref{operator_sum} from Lemma~\ref{L:op est}, we have
\begin{align*}
I_1 &\lesssim \int_\kappa^\infty  \varkappa^{2s+2} \!\!\sum_{N_1\sim N_2\leq \kappa} \tfrac{N_2}{\varkappa^3} \|q_{N_1}\|_{L^2}  \|q_{N_2}\|_{L^2} \|q\|_{L^2}^2  \,\tfrac{d\varkappa}{\varkappa} \\
    &\lesssim \kappa^{2s-1} \|q\|_{L^2}^2 \!\! \sum_{N_1\sim N_2\leq \kappa} \!\! N_2 \|q_{N_1}\|_{L^2}  \|q_{N_2}\|_{L^2} \\
    &\lesssim \kappa^{2s} \|q\|_{L^2}^4.
\end{align*}

Proceeding analogously and using \eqref{Hs high}, we find
\begin{align*}
I_2 &\lesssim  \sum_{N_1\sim N_2 >\kappa} \int_{N_2}^\infty  \eta\kappa \varkappa^{2s-1} N_2^{-2s}  \|q_{N_1}\|_{H^s}  \|q_{N_2}\|_{H^s} \|q\|_{L^2}^2  \, \tfrac{d\varkappa}{\varkappa} \\
    & \lesssim  \sum_{N_1\sim N_2>\kappa}  \eta\kappa N_2^{-1} \|q_{N_1}\|_{H^s}  \|q_{N_2}\|_{H^s} \|q\|_{L^2}^2 \\
    & \lesssim  \eta \|q\|_{L^2}^2 \beta^{[2]}_s(\kappa;q),
\\
I_3 &\lesssim \sum_{N_1\sim N_2 >\kappa} \int_{N_2}^\infty  \varkappa^{2s-1} N_2^{1-2s} \|q_{N_1}\|_{H^s}  \|q_{N_2}\|_{H^s} \|q\|_{L^2} \|q_{>\eta\kappa}\|_{L^2}  \,\tfrac{d\varkappa}{\varkappa} \\
	& \lesssim \sum_{N_1\sim N_2 >\kappa} \|q_{N_1}\|_{H^s}  \|q_{N_2}\|_{H^s} \|q\|_{L^2} \|q_{>\eta\kappa}\|_{L^2} \\
    & \lesssim   \|q\|_{L^2} \|q_{>\eta\kappa}\|_{L^2} \beta^{[2]}_s(\kappa;q), 
\\
I_4 &\lesssim \sum_{N_1\sim N_2> \kappa}  \int_\kappa^{N_2} \! \eta\kappa \varkappa^{2s}\log\left(4+\tfrac{N_2^2}{\varkappa^2}\right)N_2^{-1-2s} \|q_{N_1}\|_{H^s}  \|q_{N_2}\|_{H^s} \|q\|_{L^2}^2  \,\tfrac{d\varkappa}{\varkappa} \\
    &\lesssim  \sum_{N_1\sim N_2> \kappa}  \eta\kappa N_2^{-1} \|q_{N_1}\|_{H^s}  \|q_{N_2}\|_{H^s} \|q\|_{L^2}^2 \\
    &\lesssim \eta \|q\|_{L^2}^2 \beta^{[2]}_s(\kappa;q),
\end{align*}
and finally,
\begin{align*}
I_5 &\lesssim \sum_{N_1\sim N_2> \kappa}  \int_\kappa^{N_2} \!  \varkappa^{2s}  \log\left(4+\tfrac{N_2^2}{\varkappa^2}\right)N_2^{-2s} \|q_{N_1}\|_{H^s}  \|q_{N_2}\|_{H^s} \|q\|_{L^2} \|q_{>\eta\kappa}\|_{L^2}   \,\tfrac{d\varkappa}{\varkappa} \\
    &\lesssim  \sum_{N_1\sim N_2> \kappa}  \|q_{N_1}\|_{H^s}  \|q_{N_2}\|_{H^s} \|q\|_{L^2} \|q_{>\eta\kappa}\|_{L^2} \\
    & \lesssim  \|q\|_{L^2} \|q_{>\eta\kappa}\|_{L^2} \beta^{[2]}_s(\kappa;q).
\end{align*}

Collecting all our estimates, we conclude that
\begin{align*}
     \left| \beta_s(\kappa;q)- \beta^{[2]}_s(\kappa;q) \right| \lesssim \kappa^{2s} \|q\|_{L^2}^4 + \bigl(\eta \|q\|_{L^2}^2  + \|q\|_{L^2} \|q_{>\eta\kappa}\|_{L^2} \bigr) \beta^{[2]}_s(\kappa;q)
\end{align*}
uniformly on $Q_{**}$.  As $Q_{**}$ is $L^2$-bounded and equicontinuous, we may choose $\eta$ small and then $\kappa_1\geq \kappa_0$ large to deduce that
\begin{align}\label{almost}
    \sup_{q\in Q_{**}} \beta^{[2]}_s(\kappa;q)\lesssim \sup_{q\in Q}\beta^{[2]}_s(\kappa;q) + \kappa^{2s} \sup_{q\in Q} \|q\|_{L^2}^4\quad\text{for all}\,\,\kappa\geq \kappa_1.
\end{align}
The claim \eqref{Hs bound} now follows from \eqref{Hs high} by choosing $\kappa=\kappa_1$.

It remains to prove that $H^s$-equicontinuity for $Q$ is inherited by $Q_{**}$.  This requires a different estimate for $I_1$.  Using \eqref{hilbert_schmidt} and \eqref{operator}, we obtain
\begin{align*}
I_1 &\lesssim \sum_{N_4\leq \cdots\leq N_1\leq \kappa} \sqrt{N_3 N_4} \, \|q_{N_1}\|_{L^2}  \|q_{N_2}\|_{L^2} \|q_{N_3}\|_{L^2}  \|q_{N_4}\|_{L^2}   \int_\kappa^\infty  \varkappa^{2s-1} \,\tfrac{d\varkappa}{\varkappa} \\
    &\lesssim \kappa^{2s-4\sigma} \|q \|_{H^{s}}^4,
\end{align*}
where $\sigma =\min\{s,\frac14\}$.  Now that we know \eqref{Hs bound}, we may employ it here to deduce the following analogue of \eqref{almost}:
\begin{align}\label{almost'}
\sup_{q\in Q_{**}} \beta^{[2]}_s(\kappa;q)\lesssim \sup_{q\in Q}\beta^{[2]}_s(\kappa;q)
	+ \kappa^{2s-4\sigma} C\Bigl( \, \sup_{q\in Q} \| q\|_{L^2}^2 , \, \sup_{q\in Q} \| q\|_{H^s}^2 \Bigr)^4
\end{align}
uniformly for $\kappa\geq \kappa_1$.  As $4\sigma>2s$, equicontinuity follows by sending $\kappa\to\infty$.
\end{proof}

\section{Global well-posedness in $H^s$ for $s\geq \frac16$} \label{sec; gwp}

In order to treat the line and circle simultaneously, it is convenient to introduce
\begin{align}
    A(\kappa;q) = \alpha(\kappa;q) \quad\text{on $\R$} \qquad\text{and}\qquad A(\kappa;q) = \tfrac{1-e^{-\kappa}}{1+e^{-\kappa}}\alpha(\kappa;q) \quad\text{on $\T$.}
\end{align}
This leads to parallel leading asymptotic expansions:
\begin{align*}
    A(\kappa;q)=\tfrac{i}{2}M(q) + \tfrac{1}{4\kappa} H(q) + O\bigl(\tfrac1{\kappa^2}\bigr),
\end{align*}
as follows from Lemmas~\ref{L:dominant} and~\ref{L:sub dominant}.  This expansion is important; it guides our choice of regularized Hamiltonian flows.  We choose
\begin{align*}
    H_\kappa(q):= 4\kappa \Re A(\kappa;q),
\end{align*}
since, formally at least, $H(q)=H_\kappa(q) +O(\kappa^{-1})$, which suggests that the flow generated by $H_\kappa(q)$ approximates the \eqref{DNLS} flow as the parameter $\kappa$ diverges to infinity.

The flow generated by $H_\kappa(q)$ with respect to the Poisson structure \eqref{Poisson} is
\begin{align} \label{Hk}\tag{$H_\kappa$}
\tfrac{d}{dt} q= \left( \tfrac{\delta H_\kappa}{\delta \bar q} \right)'
   	= 2\kappa\left( \tfrac{\delta A(\kappa;q)}{\delta \bar q}+ \overline{\tfrac{\delta A(\kappa;q)}{\delta q}}\right)',
	\qtq{since} \tfrac{\delta \bar A}{\delta \bar q} = \overline{\tfrac{\delta A}{\delta q}}.
\end{align}

Our first task in this section is to prove that the $H_\kappa$ flow is well-posed on $L^2$-equicontinuous sets of Schwartz initial data satisfying \eqref{M_*}, provided $\kappa$ is chosen sufficiently large depending on the equicontinuous family; see Proposition~\ref{P:Hk wp}.  Moreover, we will show that the corresponding solutions belong to  $\Schw$ for all times.  

In Lemma~\ref{L:alpha commutes}, the $H_\kappa$ flow will be shown to conserve $M(q)$ and $\alpha(\varkappa;q)$; thus, it satisfies both the $H^s$-bounds and the $H^s$-equicontinuity guaranteed by Theorem~\ref{T:Hs conservation}.  Together with Proposition~\ref{P:Hk wp}, this immediately yields well-posedness of the $H_\kappa$ flow on $H^s$ for all $0\leq s< \frac12$ under the restriction \eqref{M_*}; see Corollary~\ref{C:Hk wellposed in Hs}.

To prove that the \eqref{DNLS} flow is well-posed in $H^s$ for $\frac16\leq s< \frac12$, it then suffices to prove that this is well approximated by \eqref{Hk} flows as $\kappa\to \infty$.  An important ingredient in our argument is the commutativity of the $H_\kappa$ and \eqref{DNLS} flows, at least on $\Schw$.  This follows from Lemma~\ref{L:alpha commutes} and the well-posedness of these flows on $\Schw$ by mimicking the arguments in \cite[\S39]{MR0997295}.  In view of this commutativity, proving convergence of the \eqref{Hk} flows to the \eqref{DNLS} flow amounts to showing that the flow generated by the difference of the Hamiltonians $H(q)-H_\kappa(q)$ converges to the identity as $\kappa\to \infty$.  This final stage of the proof will be carried out in Theorem~\ref{T:DNLS respects limits}.

In order to make sense of \eqref{Hk}, we must prove that $\alpha(\kappa;q)$ is in fact differentiable. To solve \eqref{Hk} locally in time, we further need to show that this functional derivative is itself a Lipschitz function of $q$.  These goals require us to define $\alpha(\kappa;q)$ on open sets in $L^2$, rather than merely equicontinuous sets.  The next result addresses these issues.

Here and below we write $Q_\eps$ to denote the $\eps$ neighborhood of $Q$ in the $L^2$-metric.

\begin{lemma} \label{L:alpha properties}
Let $Q$ be a bounded and equicontinuous subset of $L^2$.  Then there exist $\eps>0$ and $\kappa_0\geq 1$ so that for all $\kappa\geq \kappa_0$, $\alpha(\kappa;q)$ is a real-analytic function of $q\in Q_\eps$.  Moreover, we have the following bounds
\begin{gather}
\bigl\| \tfrac{\delta \alpha(\kappa;q)}{\delta q}\bigr\|_{H^{1}} + \bigl\| \tfrac{\delta \alpha(\kappa;q)}{\delta \bar q}\bigr\|_{H^{1}}
    	\lesssim \kappa \|q\|_{L^2} \label{delta alpha bdd}\\
\bigl\|\tfrac{\delta \alpha(\kappa;q)}{\delta q}- \tfrac{\delta \alpha(\kappa;\tilde q)}{\delta q}\bigr\|_{H^1} +
    \bigl\|\tfrac{\delta \alpha(\kappa;q)}{\delta \bar q}- \tfrac{\delta \alpha(\kappa;\tilde q)}{\delta \bar q}\bigr\|_{H^1}
    \lesssim \kappa \|q-\tilde q\|_{L^2}\label{delta alpha lipschitz}
\end{gather}
where the implicit constants depend only on $Q$.  Additionally, for every $\kappa\geq \kappa_0$ and $q\in Q_\eps$, there exists $\gamma(\kappa;q)\in H^1$ so that
\begin{align}
\label{delta alpha'}
    \bigl(\tfrac{\delta \alpha(\kappa;q)}{\delta \bar q}\bigr)' &= 2\kappa \tfrac{\delta \alpha(\kappa;q)}{\delta \bar q} -i\kappa q [\gamma(\kappa;q)+1], \\
\label{delta alpha'2}
    \bigl(\tfrac{\delta \alpha(\kappa;q)}{\delta q}\bigr)'  &=-2\kappa \tfrac{\delta \alpha(\kappa;q)}{\delta  q} +i\kappa \bar q [\gamma(\kappa;q)+1], \\
\label{gamma'}
    \gamma(\kappa;q)'  &= 2 \bar q \tfrac{\delta \alpha(\kappa;q)}{\delta \bar q} -2 q \tfrac{\delta \alpha(\kappa;q)}{\delta q}.
\end{align}
Lastly, for each integer $m\geq 0$ we have
\begin{align}
\bigl\|\bigl(\tfrac{\delta \alpha(\kappa;q)}{\delta \bar q}\bigr)' \bigr\|_{H^m}&\lesssim_m \kappa \|q\|_{H^m},\label{deriv}\\
\bigl\|\langle x\rangle^{2m}\bigl(\tfrac{\delta \alpha(\kappa;q)}{\delta \bar q}\bigr)' \bigr\|_{L^2}&\lesssim_m \kappa \|\langle x\rangle^{2m}q\|_{L^2},\label{weights}
\end{align}
uniformly for $q\in Q_\eps$ and $\kappa\geq \kappa_0$.
\end{lemma}

\begin{proof}
Proposition~\ref{P:large kappa} shows that given $\delta\in(0,1]$, there exists $\kappa_0\geq 1$ so that
\begin{align*}
\sup_{q\in Q} \sqrt{\kappa} \, \|\Lambda(q)\|_{\op}\leq \tfrac\delta4 \quad\text{uniformly for $\kappa\geq \kappa_0$}.
\end{align*}
As $\Lambda(q)$ is linear in $q$, Lemma~\ref{L:HS} allows us to deduce
\begin{align}\label{equi}
\sup_{q\in Q_\eps} \sqrt{\kappa} \, \|\Lambda(q)\|_{\op}\leq \tfrac\delta2 \quad\text{uniformly for $\kappa\geq \kappa_0$},
\end{align}
provided $\eps$ is chosen sufficiently small (depending on $\delta$).

Now we must explain how to choose $\delta$. In view of \eqref{geometric decay}, $\delta\leq 1$ guarantees that the series \eqref{alpha defn} converges on $Q_\eps$.  We place an additional requirement to aid in the proofs of \eqref{delta alpha bdd} and \eqref{delta alpha lipschitz}. From Lemma~\ref{L:H^s estimates} we find that
\begin{align*}
\|(\kappa+\partial)^{-\frac14}q (\kappa-\partial)^{-\frac14}\|_{\op} \cdot \|(\kappa-\partial)^{-\frac34}q(\kappa+\partial)^{-\frac34}\|_{\op}
	\lesssim \|q\|_{L^2} \cdot \kappa^{-\frac12} \| \Lambda(q) \|_\op .
\end{align*}
Thus, we may choose $\delta$ even smaller if necessary to ensure also that
\begin{align}\label{equi for f}
\kappa \, \bigl\|(\kappa+\partial)^{-\frac14}q (\kappa-\partial)^{-\frac14}\bigr\|_{\op}
	\cdot \bigl\|(\kappa-\partial)^{-\frac34}q(\kappa+\partial)^{-\frac34}\bigr\|_{\op}
 \leq \tfrac12
\end{align}
uniformly for $q\in Q_\eps$ and $\kappa\geq \kappa_0$.

Turning now to \eqref{delta alpha bdd}, we argue by duality.  For $f\in H^{-1}$, we have
\begin{align*}
    \Bigl\langle f,\tfrac{\delta \alpha(\kappa;q)}{\delta \bar q}\Bigr\rangle
    = \sum_{\ell\geq 0}(i\kappa)^{\ell +1}\tr\left\{ \left[ (\kappa-\partial)^{-1}q (\kappa+\partial)^{-1}\bar q \right]^{\ell} (\kappa-\partial)^{-1}q(\kappa+\partial)^{-1} \bar f\right\}.
\end{align*}
The $\ell=0$ term is readily computed exactly via Lemma~\ref{L:dominant}.  For example,
\begin{align*}
i\kappa \tr\left\{ (\kappa-\partial)^{-1}q(\kappa+\partial)^{-1} \bar f\right\}= i\kappa \bigl\langle \tfrac{1}{2\kappa+\partial}f ,q \bigr\rangle
	\quad\text{in the line case.}
\end{align*}
In either geometry, this is easily seen to satisfy the desired bound.

For $\ell\geq 1$, we employ \eqref{equi for f} and Lemma \ref{L:H^s estimates} to estimate
\begin{align*}
    &\left|(i\kappa)^{\ell +1}\tr\left\{ \left[ (\kappa-\partial)^{-1}q (\kappa+\partial)^{-1}\bar q \right]^{\ell} (\kappa-\partial)^{-1}q(\kappa+\partial)^{-1} \bar f\right\}\right|\\
    &\quad\lesssim \kappa^{\ell +1}  \|(\kappa+\partial)^{-1}\bar f (\kappa-\partial)^{-1}\|_{\hs} \|q(\kappa+\partial)^{-\frac34}\|_{\hs}^2\|(\kappa+\partial)^{-\frac14}q (\kappa-\partial)^{-\frac14}\|_{\op}^\ell\\
    &\qquad\qquad \times\|(\kappa-\partial)^{-\frac34}q(\kappa+\partial)^{-\frac34}\|_{\op}^{\ell-1} \\
    &\quad\lesssim 2^{-\ell} \kappa  \|f\|_{H^{-1}}\|q\|_{L^2}^{3},
\end{align*}
with an implicit constant independent of $\ell$.  This proves that the estimate \eqref{delta alpha bdd} holds for the $\bar q$ derivative; the bound on the $q$ derivative follows in a parallel fashion.

The proof of \eqref{delta alpha lipschitz} proceeds analogously, noting that one can always exhibit the difference $q-\tilde q$ in place of a $q$.

We define $\gamma(\kappa;q)$ via the associated linear functional,
\begin{align}\label{gamma as fnl}
    \langle f,\gamma(\kappa;q) \rangle =& \sum_{\ell\geq 1}(i\kappa)^{\ell} \tr\Bigl\{ \bigl[ (\kappa-\partial)^{-1}q (\kappa+\partial)^{-1}\bar q \bigr]^{\ell} (\kappa-\partial)^{-1} \bar f \Bigr\}\\
    &+ \sum_{\ell\geq 1}(i\kappa)^{\ell} \tr\Bigl\{ \bigl[ (\kappa+\partial)^{-1}\bar q(\kappa-\partial)^{-1}q \bigr]^{\ell} (\kappa+\partial)^{-1} \bar f \Bigr\}, \notag
\end{align}
and will prove $\gamma\in H^{1}$ by showing that this functional is bounded for $f\in H^{-1}$.

Regarding the $\ell=1$ terms, Lemma~\ref{L:H^s estimates} and direct computation show that 
\begin{align*}
\Bigl| \kappa \tr\Bigl\{ (\kappa\mp\partial)^{-1}q (\kappa\pm\partial)^{-1}\bar q (\kappa\mp\partial)^{-1} \bar f \Bigr\} \Bigr|
	&\lesssim \sqrt{\kappa} \, \| q (\kappa\pm\partial)^{-1}\bar q \|_{\hs} \|f\|_{H^{-1}} \\
	&\lesssim \sqrt{\kappa} \, \| q \|_{L^2}^2 \|f\|_{H^{-1}}.
\end{align*}
For $\ell\geq 2$, we employ Lemma~\ref{L:H^s estimates} and \eqref{equi for f} as follows:
\begin{align*}
\Bigl|  \kappa^\ell & \tr\Bigl\{ \bigl[ (\kappa\mp\partial)^{-1}q (\kappa\pm\partial)^{-1}\bar q\bigr]^{\ell} (\kappa\mp\partial)^{-1} f \Bigr\} \Bigr| \\
&\lesssim \kappa^{\ell-\frac12} \|f\|_{H^{-1}} \|q(\kappa+\partial)^{-\frac34}\|_{\hs}^2 \|(\kappa-\partial)^{-\frac12}\|_{\op}
	\|(\kappa+\partial)^{-\frac14}q (\kappa-\partial)^{-\frac14}\|_{\op}^\ell\\
&\qquad\qquad \times
	\|(\kappa-\partial)^{-\frac34}q(\kappa+\partial)^{-\frac34}\|_{\op}^{\ell-2} \\
	&\lesssim 2^{-\ell} \sqrt{\kappa}\, \| q \|_{L^2}^4 \|f\|_{H^{-1}},
\end{align*}
where the implicit constant is independent of $\ell$.  Thus $\gamma\in H^1$ and 
$$
\|\gamma(\kappa;q)\|_{H^1}\lesssim \sqrt \kappa \|q\|_{L^2}^2.
$$

The proofs of \eqref{delta alpha'} and \eqref{delta alpha'2} follow parallel arguments.  In the former case, we pair $\tfrac{\delta \alpha(\kappa;q)}{\delta \bar q}$ with $f'$, which we then rewrite as a trace.  The result then follows by noting the operator identity $f'=-(\kappa-\partial)f-f(\kappa+\partial) +2\kappa f$ and simplifying.

The proof of \eqref{gamma'} follows the same style: one pairs $\gamma(\kappa;q)$ with $f'$ and employs the operator identity $f'=[\kappa+\partial,f]=-[\kappa-\partial,f]$.

The proof of \eqref{deriv} mimics closely that of \eqref{delta alpha bdd}, once one understands how to move the derivatives from the test function $f$ to copies of $q$. Introducing the notation $f_h(x)=f(x-h)$, we observe that by the translation invariance of the trace,
\begin{align*}
&\Bigl\langle f^{(m)}, \bigl(\tfrac{\delta \alpha(\kappa;q)}{\delta \bar q}\bigr)'\Bigr\rangle 
	=  - \tfrac {\partial^m}{\partial h^m}\Big|_{h=0}\Bigl\langle f_h',\tfrac{\delta \alpha(\kappa;q)}{\delta \bar q}\Bigr\rangle
	=  - \tfrac {\partial^m}{\partial h^m}\Big|_{h=0}\Bigl\langle f',\tfrac{\delta }{\delta \bar q} \alpha(\kappa;q_{-h})\Bigr\rangle\\
	&= - \tfrac {\partial^m}{\partial h^m}\Big|_{h=0}\sum_{\ell\geq 0}(i\kappa)^{\ell +1}\tr\left\{ \!\left[ (\kappa-\partial)^{-1}q_{-h} (\kappa+\partial)^{-1}\bar q_{-h} \right]^{\ell} \!\!(\kappa-\partial)^{-1}q_{-h}(\kappa+\partial)^{-1} \bar f'\!\right\}.
\end{align*}
Next, we apply the estimates used to prove \eqref{delta alpha bdd} together with the elementary inequality
$$
\bigl\| q^{(n)}\bigr\|_{L^2} \lesssim \|q\|_{L^2}^{1-\frac nm}\|q\|_{H^m}^{\frac nm} \quad\text{for all $0< n\leq m$}.
$$ 
This yields the estimate \eqref{deriv}.  Note that summability in $\ell$ is guaranteed by \eqref{equi for f}, just as before.

Lastly, we turn to \eqref{weights}.  The argument is very similar; the key ingredient is to move the polynomial weight $\langle x\rangle^{2m}$ from the test function $f$ to a copy of $q$.  This is achieved via the identity
$$
q(\kappa+\partial)^{-1} P \bar f= \sum_{n\geq 0} (-1)^{n}[P^{(n)} q](\kappa+\partial)^{-n-1}\bar f
$$
valid for any polynomial $P(x)$, which follows easily by induction using 
\begin{equation*}
[(\kappa+\partial)^{-1}, P(x)] =  -(\kappa+\partial)^{-1}P'(x)(\kappa+\partial)^{-1}.\qedhere
\end{equation*}
\end{proof}

\begin{lemma} \label{L:alpha commutes} Let $Q\subseteq \Schw$ be $L^2$-bounded and equicontinuous, and let $\eps$ and $\kappa_0$ be as in Lemma~\ref{L:alpha properties}.  Then for all $\kappa, \varkappa \geq \kappa_0$,
\begin{align*}
    \{H,\alpha(\kappa)\}=0, \quad \{M,\alpha(\kappa)\}=0, \quad \text{and}\quad \{\alpha(\varkappa),\alpha(\kappa)\}=0 
\end{align*}
on $Q_\eps$.  Consequently, $A(\kappa)$, $A(\varkappa)$, $M$, $H$, and $H_\kappa$ all Poisson commute on $Q_\eps$.
\end{lemma}

\begin{proof}
As discussed in the introduction, the commutativity of $\alpha(\kappa)$ with the Hamiltonian $H$ was proved in \cite{klaus2020priori,tang2020microscopic} whenever the series defining $\alpha(\kappa)$ can be guaranteed to converge.  Such convergence is guaranteed by Lemma~\ref{L:alpha properties}.

Recalling \eqref{Poisson} and employing \eqref{delta alpha'}, \eqref{delta alpha'2}, and \eqref{gamma'}, we find
\begin{align*}
  \{M,\alpha(\kappa)\} = -2\kappa \int \gamma' \,dx = 0.
\end{align*}
Notice that \eqref{gamma'} guarantees $\gamma'\in L^1$.

If $\kappa=\varkappa$ the third equality is clear. When $\kappa\neq \varkappa$, we may proceed to compute the Poisson bracket by applying \eqref{delta alpha'} and \eqref{delta alpha'2} directly to the derivatives of $\alpha(\kappa)$ or by employing integration by parts and then the corresponding formulae for the partial derivatives of $\alpha(\varkappa)$.  Comparing the two approaches yields
\begin{align*}
    \varkappa \{\alpha(\kappa), \alpha(\varkappa)\} &= \kappa \{\alpha(\kappa), \alpha(\varkappa)\} \qtq{and so} \{\alpha(\varkappa),\alpha(\kappa)\}=0. \qedhere
\end{align*}
\end{proof}

\begin{proposition}\label{P:Hk wp} For each $L^2$-equicontinuous set $Q\subseteq \Schw$ satisfying \eqref{M_*}, there exists $\kappa_0\geq 1$ sufficiently large such that for all $\kappa\geq\kappa_0$, the \eqref{Hk} flow is globally well-posed for initial data in $Q$.  Moreover, the solutions remain in $\Schw$ for all time.  Lastly, the set 
$$
Q_*:=\{ e^{tJ\nabla H_\kappa} q: q\in Q, \  t\in \mathbb R, \text{ and } \kappa\geq \kappa_0\}
$$
is bounded and equicontinuous in $L^2$.
\end{proposition}

\begin{proof}
Recall the set $Q_{**}$ introduced in \eqref{D:Q*}. By \eqref{M equal}, the hypothesis \eqref{M_*}, and the definition of $M_*$, this set is bounded and equicontinuous in $L^2$. We fix $\eps>0$ and $\kappa_0\geq 1$ as the values obtained by applying Lemma~\ref{L:alpha properties} to the set $Q_{**}$.

Next we construct a local solution for initial data $q(0)\in Q_{**}$.  For $\kappa\geq \kappa_0$, Lemma~\ref{L:alpha properties} ensures that one can run the usual contraction mapping argument for the integral equation
\begin{align*}
    q(t)=q(0) +\int_0^t 2\kappa\left( \tfrac{\delta A(\kappa;q(s))}{\delta \bar q}+ \overline{\tfrac{\delta A(\kappa;q(s))}{\delta q}}\right)'\, ds
\end{align*}
to find a unique solution $q\in C([0,T]; L^2)$, provided $T$ is chosen sufficiently small.  In fact, $T$ is chosen so small that $q(t)$ and indeed all Picard iterates remain in the $\eps$-neighborhood of $Q_{**}$.

Combining the estimates \eqref{deriv} and \eqref{weights} with the Gronwall inequality shows that $q(t)\in \Schw$ for all $t\in[0,T]$.  This in turn allows us to apply Lemma~\ref{L:alpha commutes} to conclude that $\alpha(\varkappa;q(t))$ and hence $a(\varkappa;q(t))$ are conserved.  Taken together, these observations guarantee that $q([0,T])\subseteq Q_{**}$ and so the local solutions may be concatenated to yield a global solution lying wholly within $Q_{**}$.
Finally, as $Q_*$ is a subset of $Q_{**}$, it is $L^2$-bounded and equicontinuous.
\end{proof}

Combining Proposition~\ref{P:Hk wp} with Theorem~\ref{T:Hs conservation} immediately yields well-posedness of the \eqref{Hk} flow in the following sense:

\begin{corollary}\label{C:Hk wellposed in Hs}
Fix $0<s<\frac12$ and let $Q\subseteq \Schw$ be $H^s$-bounded and satisfy \eqref{M_*}.  Then there exists $\kappa_0\geq 1$ so that for all $\kappa\geq\kappa_0$ the \eqref{Hk} flow is globally well-posed for initial data in $Q$.  Moreover,
$$
Q_*:=\{ e^{tJ\nabla H_\kappa} q: q\in Q,\  t\in \mathbb R,  \text{ and } \kappa\geq \kappa_0\}\subseteq\Schw \quad\text{is $H^s$-bounded.}
$$
If $Q$ is $H^s$-equicontinuous, then so is $Q_*$.
\end{corollary}

In order to complete the proof of Theorem~\ref{T:well}, we must prove that $H^s$-Cauchy sequences of initial data $q_n(0)\in\Schw$ satisfying \eqref{M_*} lead to Cauchy sequences of solutions to \eqref{DNLS}.  As mentioned above, this will be effected by showing that the flow
\begin{align} \label{Hk_diff}\tag{$H_\kappa^{\rm{diff}}$}
    \tfrac{d}{dt} q= \left[iq'-|q|^2q- 2\kappa\left( \tfrac{\delta A(\kappa;q)}{\delta \bar q}+ \overline{\tfrac{\delta A(\kappa;q)}{\delta q}}\right)\right]',
\end{align}
generated by $H(q)-H_\kappa(q)$, converges to the identity as $\kappa\to\infty$.  Due to commutativity of the flows, $\Schw$-valued solutions to \eqref{Hk_diff} can be built via
$$
e^{tJ\nabla H_\kappa^{\rm{diff}}}q= e^{tJ\nabla H}  e^{-tJ\nabla H_\kappa} q
$$
using Corollaries~\ref{C:H1} and \ref{C:Hk wellposed in Hs}. In view of Lemma~\ref{L:alpha commutes}, these solutions conserve $M$ and $\alpha(\varkappa)$.

The proof of our final theorem makes a fitting end for this paper by highlighting the power of equicontinuity.  It is also here that we will finally see the origin of the restriction $s\geq \frac16$.  It is needed to make sense of the nonlinearity in \eqref{Hk_diff} pointwise in time.

\begin{theorem} \label{T:DNLS respects limits}
Fix $\tfrac{1}{6} \leq s<\tfrac{1}{2}$ and $T>0$. Given a sequence $q_n(0)\in \Schw$ of initial data that converges in $H^s$ and satisfies \eqref{M_*},  let $q_n(t)$ denote the corresponding solutions to \eqref{DNLS}. Then $q_n(t)$ converges in $H^s$, uniformly for $|t|\leq T$.
\end{theorem}

\begin{proof}
By hypothesis, the set $Q=\{q_n(0): n\in\mathbb N\}$ is bounded and equicontinuous in the $H^s$-metric.  Let $\kappa_0\geq 1$ be as given by Corollary~\ref{C:Hk wellposed in Hs}.  Then for $\kappa\geq\kappa_0$, the \eqref{Hk} flow is well-posed for initial data in $Q$ and the set 
$$
Q_*:=\{e^{tJ\nabla H_\kappa} q_n(0): n\in\mathbb N,\  t\in\mathbb R, \text{ and } \kappa\geq \kappa_0 \}\subseteq\Schw
$$ 
is bounded and equicontinuous in $H^s$.

The commutativity of the \eqref{Hk} and the \eqref{DNLS} flows allows us to rewrite our sequence of solutions as
\begin{align*}
    q_n(t)=e^{tJ\nabla H_\kappa^{\rm{diff}}} e^{tJ\nabla H_\kappa} q_n(0).
\end{align*}
Moreover, by Theorem~\ref{T:Hs conservation}, the set
\begin{align*}
        \{e^{tJ\nabla H_\kappa^{\rm{diff}}} q: q\in Q_*,\  t\in\mathbb R, \text{ and } \kappa\geq \kappa_0 \}\subseteq Q_{**}
\end{align*}
is bounded and equicontinuous in $H^s$. 

We will show that $q_n(t)$ forms a Cauchy sequence in $H^s$, uniformly for $|t|\leq T$. By the definition of $Q_*$, we estimate
\begin{align*}
    \sup_{|t|\leq T}\|q_n(t)-q_m(t)\|_{H^s}\leq &  2  \sup_{q\in Q_*}\sup_{|t|\leq T}\|e^{tJ\nabla H_\kappa^{\rm{diff}}} q-q\|_{H^s}\\
    &+ \sup_{|t|\leq T}\|e^{tJ\nabla H_\kappa} q_n(0) - e^{tJ\nabla H_\kappa} q_m(0)\|_{H^s}
\end{align*}
for all $\kappa\geq \kappa_0$. For any such fixed $\kappa$, the well-posedness of the \ref{Hk} flow ensures that the last term of the right-hand side converges to $0$ as $n, m\to\infty$.  Thus, it suffices to prove that the difference flow converges to the identity uniformly on $Q_*$:
\begin{align*}
    \lim_{\kappa\to\infty}\sup_{q\in Q_*}\sup_{|t|\leq T}\|e^{tJ\nabla H_\kappa^{\rm{diff}}} q-q\|_{H^s}=0.
\end{align*}
In fact, as $Q_{**}$ is $H^s$-equicontinuous, it suffices to show that
\begin{align}\label{diff converges to id}
    \lim_{\kappa\to\infty}\sup_{q\in  Q_*}\sup_{|t|\leq T}\|e^{tJ\nabla H_\kappa^{\rm{diff}}} q-q\|_{H^{-4}}=0.
\end{align}

By the fundamental theorem of calculus and \eqref{Hk_diff}, proving \eqref{diff converges to id} reduces to showing that
\begin{align}\label{diff converges to id'}
\lim_{\kappa\to\infty}\sup_{q\in  Q_{**}} \| F \|_{H^{-3}}=0
\qtq{where} 
F := iq'-|q|^2q- 2\kappa\left( \tfrac{\delta A(\kappa;q)}{\delta \bar q}+ \overline{\tfrac{\delta A(\kappa;q)}{\delta q}}\right)  .\!\!
\end{align}
A straightforward computation shows that $F^{[1]}$, the term in $F$ that is linear in $q$, is given by $-i\tfrac{\partial^3}{4\kappa^2- \partial^2}q$.  This clearly converges to zero in $H^{-3}$ as $\kappa\to \infty$, uniformly on $Q_{**}$, or indeed, on any $L^2$-bounded set.

We turn now to the contribution of $F^{[3]}$, the term in $F$ that is cubic in $q$.  Employing Lemma~\ref{L:sub dominant}, we find the cubic terms \begin{align*}
\left( \tfrac{\delta A(\kappa;q)}{\delta \bar q} \right)^{[3]}  &=-  \tfrac{\kappa^2}{2\kappa-\partial} \left\{\left(\tfrac{1}{2\kappa+\partial} \bar q \right)(4\kappa-\partial)\left(\tfrac{1}{2\kappa-\partial} q\right)^{2}\right\}\\
    &=-  \tfrac{2\kappa^2}{2\kappa-\partial} \left\{q\left(\tfrac{1}{2\kappa-\partial} q\right) \left(\tfrac{1}{2\kappa+\partial} \bar q \right)\right\},\\
\left( \overline{\tfrac{\delta A(\kappa;q)}{\delta q}} \right)^{[3]}  &=-  \tfrac{\kappa^2}{2\kappa+\partial} \left\{ \left(\tfrac{1}{2\kappa-\partial} \bar q \right)(4\kappa+\partial)\left(\tfrac{1}{2\kappa+\partial} q\right)^2\right\}\\
     &=-  \tfrac{2\kappa^2}{2\kappa+\partial} \left\{q\left(\tfrac{1}{2\kappa+\partial} q\right) \left(\tfrac{1}{2\kappa-\partial} \bar q \right)\right\}.
\end{align*}
This allows us to compute the full cubic term as follows:
\begin{align*}
F^{[3]}
    &=2\kappa^2 \tfrac{\partial}{2\kappa-\partial} \left[ q\left(\tfrac{1}{2\kappa-\partial} q\right) \left(\tfrac{1}{2\kappa+\partial} \bar q \right)\right] -2\kappa^2 \tfrac{\partial}{2\kappa+\partial} \left[ q\left(\tfrac{1}{2\kappa+\partial} q\right) \left(\tfrac{1}{2\kappa-\partial} \bar q \right)\right]\\
    & \qquad +2\kappa^2  q\left(\tfrac{1}{2\kappa-\partial} q\right) \left(\tfrac{1}{2\kappa+\partial} \bar q \right) +2\kappa^2  q\left(\tfrac{1}{2\kappa+\partial} q\right) \left(\tfrac{1}{2\kappa-\partial} \bar q \right)-q^2 \bar q \\
    & =\tfrac{2\partial}{2\kappa-\partial} \left[ q\left(\tfrac{\kappa}{2\kappa-\partial} q\right) \left(\tfrac{\kappa}{2\kappa+\partial} \bar q \right)\right] -\tfrac{2\partial}{2\kappa+\partial} \left[ q\left(\tfrac{\kappa}{2\kappa+\partial} q\right) \left(\tfrac{\kappa}{2\kappa-\partial} \bar q \right)\right]
    +q^2 \left(\tfrac{\partial^2} {4\kappa^2-\partial^2} \bar q \right) \\
    &\qquad +  q \bar q\left(\tfrac{\partial^2} {4\kappa^2-\partial^2} q\right)- \tfrac{1}{2}q \left(\tfrac{\partial}{2\kappa-\partial} q\right) \left(\tfrac{\partial}{2\kappa+\partial} \bar q\right) -\tfrac12 q \left(\tfrac{\partial}{2\kappa+\partial} q\right) \left(\tfrac{\partial}{2\kappa-\partial} \bar q\right).
\end{align*}

To estimate its contribution, we pair with $f\in H^3$ and apply H\"older's inequality.  Boundedness is easily deduced from
\begin{gather}
\| f\|_{L^\infty} + \kappa \bigl\| \tfrac{\partial}{2\kappa\pm\partial} f \bigr\|_{L^\infty} \lesssim \| f\|_{H^3}, \label{fest}\\
\| q \|_{L^3} + \bigl\| \tfrac{\partial}{2\kappa\pm\partial} q \bigr\|_{L^3} + \bigl\| \tfrac{\partial^2}{4\kappa^2-\partial^2} q \bigr\|_{L^3}
	\lesssim \|q\|_{H^s}. \label{quest} 
\end{gather}
Evidently \eqref{quest} requires $s\geq \frac16$.  The gain of a power of $\kappa$ in \eqref{fest} guarantees that the contribution of the first two terms in $F^{[3]}$ decays to zero as $\kappa\to\infty$.  For the remaining terms, we use $H^s$-equicontinuity to obtain decay: as $s\geq \frac16$, we have that
$$
\lim_{\kappa\to\infty}\sup_{q\in Q_{**}} \bigl\|\tfrac{\partial}{2\kappa\pm\partial}q\bigr\|_{L^3}
\lesssim \lim_{\kappa\to\infty}\sup_{q\in Q_{**}} \bigl\|\tfrac{\partial}{2\kappa\pm\partial}q\bigr\|_{H^s}=0.
$$

Finally, we turn our attention to the remaining terms (quintic and higher) in the series expansion of $F$.  By Lemma~\ref{L:HS}, \eqref{alpha defn}, and the embedding $H^3\hookrightarrow L^\infty$, 
\begin{align*}
\biggl|  \int f F^{[\geq 5]} \,dx \biggr| &\lesssim \sum_{\ell\geq 2}\kappa^{\ell +2} \|\Lambda(q)\|_{\hs}^2 \|\Lambda(q)\|_{\op}^{2\ell-1} \|\Lambda(f)\|_{\op}\\
    &\lesssim \|q\|_{L^2}^2\|f\|_{H^3}\sum_{\ell\geq 2}\kappa^{\ell} \|\Lambda(q)\|_{\op}^{2\ell-1}.
\end{align*}
The convergence we require does not follow from Proposition~\ref{P:large kappa}; we would lose by a factor of $\sqrt{\kappa}$. However arguing in the same fashion, we find
\begin{align*}
\|\Lambda(q)\|_{\op}\lesssim \|\Lambda(q_{\leq\eta\kappa})\|_{\op}+ \|\Lambda(q_{>\eta\kappa})\|_{\op}&\lesssim \kappa^{-1}\|q_{\leq \eta\kappa}\|_{L^\infty}+\kappa^{-\frac 12} \|q_{>\eta\kappa}\|_{L^2}\\
    &\lesssim \kappa^{-\frac12-s}\left(\eta^{\frac 12-s} \|q\|_{H^s} +\eta^{-s} \|q_{>\eta\kappa}\|_{H^s}\right),
\end{align*}
for any $\eta>0$.  When $s>\frac16$, we may simply take $\eta=1$ to deduce that
$$
\lim_{\kappa\to\infty}\sup_{q\in  Q_{**}} \bigl\| F^{[5]} \bigr\|_{H^{-3}}=0.
$$
For the endpoint case $s=\frac16$, this follows from the $H^s$-equicontinuity of $Q_{**}$ by choosing $\eta$ small and then $\kappa$ large.
This completes the proof of the theorem.
\end{proof}

\bibliography {bibliography}
\bibliographystyle{amsplain}

\end{document}